\documentclass[a4paper,11pt,reqno]{amsart}

\usepackage{
amsfonts,
amsmath,
amssymb,
bbold,
dsfont,
mathrsfs,
soul,
subfig,
graphicx,
xcolor
}
\usepackage[colorlinks=true,urlcolor=blue, citecolor=red,linkcolor=blue,
linktocpage,pdfpagelabels, bookmarksnumbered,bookmarksopen]{hyperref}

\usepackage{geometry}
\usepackage[latin1]{inputenc}
\usepackage{mathtools,
mathrsfs, 
amsopn,
amsthm}
\mathtoolsset{showonlyrefs}

\newtheorem{teo}{Theorem}[section]
\newtheorem{lem}[teo]{Lemma}
\newtheorem{pro}[teo]{Proposition}
\newtheorem{cor}[teo]{Corollary}

\theoremstyle{remark}
\newtheorem{rem}[teo]{Remark}

\theoremstyle{definition}

\newcommand{\R}{\mathbb{R}}
\newcommand{\N}{\mathbb{N}}
\newcommand{\C}{\mathbb{C}}
\newcommand{\Z}{\mathbb{Z}}

\newcommand{\ham}{\mathcal{H}_{\alpha(t)}}
\newcommand{\hamz}{\mathcal{H}_{\alpha(0)}}
\newcommand{\dom}{\mathcal{D}}
\newcommand{\gl}{\mathcal{G}_\lambda}

\newcommand{\g}{\mathcal{G}}
\renewcommand{\S}{\mathcal{S}}
\newcommand{\form}{\mathcal{F}_{\alpha(t)}}
\newcommand{\domf}{\mathcal{V}}
\newcommand{\ID}{\mathbb{I}}
\newcommand{\re}{\mathrm{Re}}
\newcommand{\im}{\mathrm{Im}}
\newcommand{\NN}{\mathcal{N}}

\newcommand{\A}{\mathcal{A}}
\newcommand{\B}{\mathcal{B}}
\newcommand{\I}{\mathcal{I}}

\newcommand{\psitr}{\widetilde{\psi}}
\newcommand{\phitr}{\widetilde{\phi}}
\newcommand{\vphitr}{\widetilde{\varphi}}

\newcommand{\xv}{\mathbf{x}}
\newcommand{\yv}{\mathbf{y}}
\newcommand{\zev}{\mathbf{0}}
\newcommand{\kv}{\mathbf{k}}

\newcommand{\dyv}{\,\mathrm{d}\yv}
\newcommand{\dkv}{\,\mathrm{d}\kv}

\newcommand{\dt}{\,\mathrm{d}t}
\newcommand{\dtau}{\,\mathrm{d}\tau}
\newcommand{\ds}{\,\mathrm{d}s}
\newcommand{\drho}{\,\mathrm{d}\varrho}
\newcommand{\dP}{\,\mathrm{d}p}
\newcommand\cL{{\mathcal{L}}}

\newcommand{\ep}{\varepsilon}
\newcommand{\al}{\alpha}
\renewcommand{\th}{\theta}
\newcommand{\la}{\lambda}
\newcommand{\const}{\mathcal{C}}
\newcommand{\clog}{C_{\mathrm{log},\beta}}

\newcommand{\one}{\mathds{1}}
\newcommand{\IN}[1]{\mathrm{int}\big(#1\big)}

\newcommand{\braket}[2]{\left\langle #1,#2 \right\rangle}
\newcommand{\f}[2]{\frac{#1}{#2}}
\newcommand{\tf}[2]{\tfrac{#1}{#2}}

\begin{document}

\title[Complete ionization for a non-autonomous point interaction model in $d=2$]{Complete ionization for a non-autonomous point interaction model in $d=2$}

\author[W. Borrelli]{William Borrelli}
\address[W. Borrelli]{Dipartimento di Matematica e Fisica, Universit\`a Cattolica del Sacro Cuore, Via dei Musei 41, I-25121, Brescia, Italy.} 
\email{william.borrelli@unicatt.it}

\author[R. Carlone]{Raffaele Carlone}
\address[R. Carlone]{Universit\`{a} ``Federico II'' di Napoli, Dipartimento di Matematica e Applicazioni ``R. Caccioppoli'', MSA, via Cinthia, I-80126, Napoli, Italy.}
\email{raffaele.carlone@unina.it}

\author[L. Tentarelli]{Lorenzo Tentarelli}
\address[L. Tentarelli]{Politecnico di Torino, Dipartimento di Scienze Matematiche ``G.L. Lagrange'', Corso Duca degli Abruzzi 24, 10129, Torino, Italy.} 
\email{lorenzo.tentarelli@polito.it}

\date{\today}


\begin{abstract}  We consider the two dimensional Schr\"odinger equation with a time dependent point interaction, which represents a model for the dynamics of a quantum particle subject to a point interaction whose strength varies in time. First, we prove global well-posedness of the associated Cauchy problem under general assumptions on the potential and on the initial datum. Then, for a monochromatic periodic potential (which also satisfies a suitable no-resonance condition) we investigate the asymptotic behavior of the survival probability of a bound state of the time-independent problem. Such probability is shown to have a time decay of order $\mathcal O((\log t/t)^2)$, up to lower order terms.

\end{abstract}
\maketitle

{\footnotesize

\medskip

}



\section{Introduction}
In this paper we consider the two dimensional Schr\"odinger equation with time dependent point interaction, namely
\begin{equation}
 \label{eq:formal_equation}
 \imath\f{\partial\psi}{\partial t}=H(t)\psi,
\end{equation}
where, at any fixed time $t$, the operator  $H(t)$ is formally given by 
\begin{equation}
\label{eq:formal_point_interaction}
``H(t):=-\Delta+\alpha(t)\delta(\mathbf{x})\text{''},
\end{equation}
with $\alpha:\R\to\C$. The precise definition of the Hamiltonian $H(t)$ is given in Section \ref{sec:model} (see also Remark \ref{rem:delta}).
\medskip

Our main purpose is to establish complete ionization for the evolution of \eqref{eq:formal_equation}. More precisely, we start by proving (Theorem \ref{teo:wp}) global well-posedness of the Cauchy problem associated with \eqref{eq:formal_equation}, under general assumptions on $\alpha$ and on the initial datum. Then, in the so called \emph{monochromatic} case (i.e., $\alpha(t)$ as in \eqref{eq:monochrom}), we investigate the asymptotic behavior of the \emph{survival probability} of the $L^2$-normalized bound state associated with the sole eigenvalue of $H(0)$, denoted throughout by $\varphi_\alpha$. That is, we study the behavior of $\left|\braket{\varphi_{\alpha}}{\psi(t,\cdot)}\right|^2$, as $t\to+\infty$, where $\braket{\cdot}{\cdot}$ denotes the scalar product of $L^2(\R^2)$ and $\psi(t,\cdot)$ is the solution of \eqref{eq:formal_equation} with $\psi(0,\cdot)=\varphi_\alpha$. In particular, we show that the survival probability vanishes, as $t\to+\infty$, and establish its decay rate up to lower order terms (Theorem \ref{teo:cion}).
\medskip

The interest of the model under investigation is twofold. Time-dependent point interactions belong to a special class of perturbative models for which a completely rigorous analysis is possible; that is, the behavior of the survival probability can be estimated without a perturbative expansion (see, e.g., \cite{CCL} and references therein). On the other hand, the model under study provides an effective description of the microscopic dynamics of a quantum particle interacting with bosonic scalar quantum fields, in the so called \emph{quasi-classical limit} (see \cite{CCFO,CFO}). More precisely, in \cite{CCFO} the authors prove that time-dependent point interactions \eqref{eq:formal_equation} can be derived from a microscopic model for the dynamics of a quantum particle interacting with bosonic scalar quantum fields, in configurations where the fields are very intense and the average number of carriers is large. As explained in \cite{CCFO}, considering the coupling of the particle with two distinct species of fields is a situation not uncommon both in condensed matter physics and at high energies. A typical example is given by the description of a quantum particle coupled with both acoustic and optical phonons in a compound ionic crystal (see, e.g., \cite[ Chapter 4]{Ki}). Physically, the presence of two different species of phonons, having different scaling properties and dispersion relations, induces the ionization of the particle. In mathematical terms, the resulting model consists of a time-dependent point interaction, rather than a time-independent one. In view of this, our Theorem \ref{teo:cion} rigorously settles such result in the two dimensional case.

\medskip
Both the global well-posedness and the asymptotic behavior of the survival probability of the bound states of the Hamiltonian at $t=0$ have been widely investigated in dimensions one and three. We mention \cite{CCL,CCLR,CLR,RL} for the one-dimensional problem and \cite{CDFM,SY} for the three-dimensional one. On the contrary, to the best of our knowledge, the two dimensional case has never been discussed before in the literature. The main reason for this might be related to some technical difficulties, arising in the integral formulation of the problem, that are deep-seated features of the dimension two. However, some recent works on the nonlinear analogue of \eqref{eq:formal_equation} (where $\alpha(\cdot)$ is replaced by $|\psi|^{p}$), shed some new light on the above mentioned technical questions \cite{ACCT,ACCT2,CCT,CF} thus allowing to deal with the two-dimensional case.

More in detail, in the former part of the paper we extend the tools developed for the nonlinear problem in \cite{CCT} to the time-dependent case in order to prove global well-posedness in a weak sense; whereas, in the latter part, extending to the two-dimensional case the techniques of \cite{CCL}, we investigate the behavior of the survival probability of a bound state of $H(0)$ (in the monochromatic case) proving the so-called \emph{complete ionization}, that is, the fact that such probability tends to zero as $t\to\infty$. Note that both parts of the paper are not straightforward adaptations of already developed techniques. Indeed, on one hand the fact that \eqref{eq:formal_equation} is not autonomous, in contrast to the nonlinear problem treated in \cite{CCT}, requires major modifications of the strategy for global well-posedness. On the other hand, dimension two yields a different singularity of the kernel of the integral equation associated to \eqref{eq:formal_equation}, that 
calls for some new ides in the classical discussion of the equation in Laplace transform compared to, for instance, \cite{CDFM,CCL}.


\subsection{The model}\label{sec:model}

Before stating the main results of the paper, it is necessary to give a precise meaning to \eqref{eq:formal_equation}, that is to rigorously define \eqref{eq:formal_point_interaction}. This is done extending the definition of a time-independent delta interaction present, e.g., in \cite{AGH-KH}.

\medskip
Let $\alpha:\R\to\C$ be a suitably smooth function and define the family of operators $(\ham)_{t\in\R}$ such that for every fixed $t\in\R$, $\ham:L^2(\R^2)\to L^2(\R^2)$ is the operator with domain
\begin{multline}
 \label{eq-domain}
 \dom(\ham) :=\\[.2cm]
 =\big\{\psi\in L^2(\R^2):\exists q\in\C\:\text{ s.t. }\: \psi-q\gl=:\phi_{\lambda}\in H^2(\R^2)\:\text{ and }\:\phi_{\lambda}(\zev) = (\al(t)+\th_\la)\,q\big\},
\end{multline}
and action
\begin{equation}
\label{eq-action}
 \ham\psi:=-\Delta\phi_\la-\lambda q\gl,\qquad\forall\psi\in\dom(\ham),
\end{equation}
where $\lambda$ is a positive parameter, $\th_\la:=\f{\log(\f{\sqrt{\la}}{2})+\gamma}{2\pi}$ with $\gamma$ the Euler-Mascheroni constant, and $\gl$ is the Green's function of $-\Delta+\lambda$ on $\R^2$, namely $\gl(\xv)=\f{1}{2\pi}K_0(\sqrt{\la}|\xv|)$ with $K_0$ the modified Bessel function of second kind of order $0$ (also known as \emph{Macdonald function}, \cite[Sect. 9.6]{AS}).
Note that in the following $\phi_\lambda$ and $q\gl$ are referred to as the \emph{regular part} and \emph{singular part}, respectively, of an element in $ \dom(\ham)$.
\smallskip

Some comments are now in order. First, it is well known (see again \cite{AGH-KH}) that, for any fixed $t\in\R$, $\ham$ is the unique two-dimensional self-adjoint realization of a perturbation of $\delta$-type of $-\Delta$ with strength $\alpha(t)$. On the other hand, one can easily check that the definition of $\ham$ is independent of $\lambda$. Indeed, if there exists a representation of a function $\psi\in\dom(\ham)$ for one value of $\lambda>0$, then one can find another equivalent representation for any other $\lambda>0$ (due to the fact that $\g_{\lambda_1}-\g_{\lambda_2}\in H^2(\R^2)$, for every $\lambda_1,\lambda_2>0$). 

An even clearer way to see the independence of $\lambda$ is that of rewriting \eqref{eq-domain} and \eqref{eq-action} as follows:
\begin{multline*}
 \dom(\ham) := \\[.2cm]
 =\big\{\psi\in L^2(\R^2):\exists q\in\C\:\text{ s.t. }\:\psi-q\g_0=:\phi\in H^1_{loc}(\R^2)\cap\dot{H}^{2}(\R)\:\text{ and }\:\phi(\zev) =  \al(t)q\big\},
\end{multline*}
\[
 \ham\psi:=-\Delta\phi,\qquad\forall\xv\neq\zev,\qquad\forall\psi\in\dom(\ham),
\]
where $\g_0:=-\f{\log|\xv|}{2\pi}$ is the Green function of $-\Delta$ in $\R^2$. This formulation is equivalent to the previous one but is not useful in the computations since here the \emph{regular part} is too rough due to the \emph{infrared singularity} of $\g_0$ at infinity. For this reason, \eqref{eq-domain}-\eqref{eq-action} are usually preferred.

Given the definition of $\ham$, it is possible to give a precise meaning to \eqref{eq:formal_equation} and to consider the associated Cauchy problem 
\begin{equation}
 \label{eq:LSE}
 \left\{
 \begin{array}{l}
  \displaystyle \imath\f{\partial\psi}{\partial t}=\ham\psi\\[.4cm]
  \displaystyle \psi(0,\cdot)=\psi_0,
 \end{array}
 \right.
\end{equation}


\subsection{Main results}
\label{subsec:main}

Even though $\ham$ is self-adjoint at any fixed time $t$, \eqref{eq:LSE} cannot be solved by means of Stone's Theorem as the operator $\ham$ depends itself on $t$. As a consequence, the first point to be discussed is the global well-posedness of \eqref{eq:LSE}. Unfortunately, due to technical reasons (see, e.g., Remark \ref{rem:regsob}), we are not able to prove strong global well-posedness for \eqref{eq:LSE}. Nevertheless, it is possible to prove weak global well-posedness, which is sufficient for the purposes of the paper.

In order to define the weak version of \eqref{eq:LSE}, it is necessary to introduce the quadratic form associated with $\ham$, with domain
\begin{equation}
 \label{eq:form_domain}
 \domf:=\left\{ \psi \in L^2(\R^2): \psi-q\gl=:\phi_\la \in H^1(\R^2),\,q\in \C \right\}
\end{equation}
and action
\[
 \form(\psi):= \|\nabla\phi_\la\|_{L^2(\R^2)}^2 + \la\big(\|\phi_\la\|_{L^2(\R^2)}^2-\|\psi\|_{L^2(\R^2)}^2\big)+(\al(t)+\th_\la)\,|q|^2.
\]
Note that, in contrast to $\dom(\ham)$, $\domf$ is time-independent as the so called \emph{boundary condition} $\phi_\la(0)=(\al(t)+\th_\la)\,q$ is now absent. Therefore, the weak formulation of \eqref{eq:LSE} is given by
\begin{equation}
 \label{eq:weak_LSE}
 \begin{cases}
  \displaystyle \imath \frac{d}{dt} \braket{\chi}{\psi(t,\cdot)}=\langle\chi, \psi(t,\cdot)\rangle_{\form},\quad\forall\chi \in \domf,\\[.3cm]
  \displaystyle \psi(0,\cdot) = \psi_0,
 \end{cases}
\end{equation}
where $\langle\cdot, \cdot\rangle_{\form}$ is the sesquilinear form associated with $\form$.

Now, before stating the first main result of the paper it is convenient to recall some other tools which are widely used throughout. First, we denote by $U(t):=e^{ i \Delta t}$ the two-dimensional free Schr\"odinger propagator, with kernel
\begin{equation}\label{eq:Schrodingerkernel}
 U(t,\xv) := \frac{e^{-\frac{|\xv|^{2}}{4 \imath t}}}{2 \imath t},\qquad t \in \R,\:\xv \in \R^2,
\end{equation}
(acting by the normalized convolution product $(f*g)(\xv):=\f{1}{2\pi}\int_{\R^2}f(\xv-\yv)g(\yv)\dyv$). Moreover, we denote by $\I$ the \emph{Volterra function} of order $-1$
\begin{equation}
 \label{eq:i_kernel}
 \I(t) : = \int_{0}^{+\infty}\frac{t^{\tau - 1}}{\Gamma(\tau)}\dtau,
\end{equation}
where $\Gamma$ represents, as usual, the Euler gamma function \cite{CCT,CF}. Observe that \eqref{eq:i_kernel} is a particular case of the Volterra $\mu$-functions
\begin{equation}
 \label{eq:gen_volt}
 \mu(t,\beta,\delta):=\int_0^{+\infty}\frac{t^{\delta+s}\,s^{\beta}}{\Gamma(\beta+1)\,\Gamma(\delta+s+1)}\ds
\end{equation}
 cfr. \cite[Section 18.3]{E}; precisely, $\I(t)=\mu(t,0,-1)$. 

 \medskip
Then, the global well-posedness can be stated as follows.

\begin{teo}[Global well-posedness]
 \label{teo:wp}
 Let $(\ham)_{t\in\R}$ be the family of operators given by \eqref{eq-domain} and \eqref{eq-action}, with $\alpha\in W_{loc}^{1,\infty}([0,+\infty))$, and let $ \psi_0 \in \dom $, with
 \begin{equation}
  \label{eq:extradom}
  \dom := \left\{\psi\in\domf:\phi_\lambda \in H^2(\R^2)\right\}.
 \end{equation}
 Then, there exists a unique function $\psi$ such that $\psi(t,\cdot)\in\domf$ and satisfies \eqref{eq:weak_LSE}, for every $ t >0 $. In addition, $\psi$ is given by
 \begin{equation}
  \label{eq:ansatz}
  \psi(t,\xv) := (U(t)\psi_0)(\xv)+\frac{\imath}{2\pi}\int_0^t U(t-\tau,\xv)\,q(\tau)\dtau,
 \end{equation}
 where $ q(\cdot)$ is the unique continuous solution of
 \begin{equation}
 \label{eq:charge_eq}
  q(t)+4\pi\int_0^t\I(t-\tau)\big(\al(\tau)+\th_\lambda-\tfrac{\imath}{8})\big)q(\tau)\dtau=4\pi\int_0^t\I(t-\tau)(U(\tau)\psi_0)(\zev)\dtau.
 \end{equation}
 Finally, the \emph{mass} $ M(t)=M\big(\psi(t,\cdot)\big):=\left\| \psi(t,\cdot)\right\|^2_{L^2(\R^2)} $ is preserved along the flow.
\end{teo}

Some comments are in order. First, in the sequel we will refer to the function $q(t)$ as to the \emph{charge} and, consequently, to equation \eqref{eq:charge_eq} as to the \emph{charge equation}.

On the other hand, let us mention that Theorem \ref{teo:wp} can be established even under slightly weaker, but more technical, assumptions on the initial datum $\psi_0$ (see Remark \ref{rem:dato}). However, at the moment the natural hypothesis on the initial datum, i.e. $\psi_0\in\domf$, does not allow to prove local and global well posedness. In other words, techniques developed thus far do not enable one to prove that the form domain $\domf$ is preserved along the flow (see again Remark \ref{rem:dato}). Hence, we decided to use assumption \eqref{eq:extradom} for the sake of simplicity. Observe that since we are not able to prove the existence of a time-independent domain, which is dense in $L^2(\R^2)$ and preserved by the flow, we cannot deduce the existence of a unitary two-parameters propagator generated by $\ham$.

Finally, the proof of strong local and global well posedness in the operator domain $\dom(\ham)$ is out of reach as well at the moment. The reason is manily due to the highly singular behavior of the integral kernel $\I$ of \eqref{eq:charge_eq}, which not only misses any Sobolev regularizing effect, but also causes a net loss of Sobolev regularity when the Sobolev index is large (see Remark \ref{rem:dominio}). 
\medskip

Once global well-posedness of \eqref{eq:weak_LSE} is established, one can focus on the asymptotic analysis of the survival probability of the bound state associated to the sole eigenvalue of $\hamz$:
\begin{equation}\label{eq:eigenv}
 \lambda_0:=-4 e^{-4\pi\alpha(0)-2\gamma}
\end{equation}
(recall that here, in contrast to what occurs in dimensions one and three, $\hamz$ has an eigenvalue independently of the sign of $\alpha(0)$). Throughout we only consider the real $L^2$-normalized bound state associated to $\la_0$, namely
\begin{equation}
 \label{eq:eigen}
  \varphi_0:=\const_{0}\mathcal{G}_{-\lambda_0},\qquad\text{with}\quad\const_{0}:=2\sqrt{-\pi\lambda_0}.
 \end{equation}
Note that clearly $\varphi_0\in\dom$ so that it is an initial datum for which our arguments work .

Now, before stating the main theorem it is necessary to introduce the following assumption on the strength of the singular perturbation.

\smallskip
\noindent\textbf{Monochromatic assumption.} We say that the family of operators $(\ham)_{t\in\R}$ given by \eqref{eq-domain} and \eqref{eq-action} is a \emph{monochromatic perturbation} of the Laplacian if
\begin{equation}
\label{eq:monochrom}
\alpha(t)=\alpha_0\sin(\omega t+\eta)+c,\qquad\forall t\in\R,
\end{equation}
with $\alpha_0\in\R\setminus\{0\}$, $\omega>0$ and $\eta,c\in\R$.
\smallskip

\begin{rem}\label{rem:delta}
We mention that, even when $\alpha_0\sin\eta+c=0$, the operator $\mathcal{H}_{\alpha(0)}$ is \emph{not} the free Laplacian and thus $\ham$ is a non-trivial perturbation of the Laplacian even for $t=0$ (for details see \cite[Ch. I.5]{AGH-KH}). In view of this, the formal notation in \eqref{eq:formal_point_interaction} might be a little bit misleading, but we use it nevertheless throughout the paper as it does not give rise to misunderstandings.
\end{rem}

\begin{teo}[Complete ionization in the non-resonant case]
 \label{teo:cion}
 Let $(\ham)_{t\in\R}$ be a family of operators fulfilling the monochromatic assumption and let $\varphi_0 $ be the $L^2$-normalized bound state of $\hamz$ (see \eqref{eq:eigen}). Assume also that 
 \begin{equation}
  \label{eq:nonres}
 \omega e^{-2(\log2-\gamma)+4\pi\alpha(0)}\notin\mathbb N\,.
 \end{equation}
 If $\psi$ is the solution of \eqref{eq:weak_LSE} with $\psi_0=\varphi_0$ (provided by Theorem \ref{teo:wp}  and $\Theta(t)$ is the survival amplitude of $\varphi_0$, i.e.
 \begin{equation}
  \label{eq:survival}
  \Theta(t):=\braket{\varphi_{0}}{\psi(t,\cdot)},
 \end{equation}
 then
 \begin{equation}
  \label{survivalfinal}
 \Theta(t)=\braket{\varphi_0}{U(t)\varphi_0}+L(t)+R(t),\qquad\forall t>0,
\end{equation}
where $L(\cdot),\,R(\cdot)$ satisfy
\begin{equation}
\label{eq:decay}
\displaystyle
 L(t)\sim \frac{C_1+C_2\log t}{t}\,\qquad \text{and} \qquad |R(t)|\leq Be^{-b t}\,,\qquad\text{as}\quad t\to+\infty,
\end{equation}
with $C_1,\,C_2\in\C$ and $B,b>0$ fixed.
\end{teo}

\begin{rem}
 Note that the constants $C_1,\,C_2$ cannot be proved to be different from zero in general. Indeed, as explained by Remark \ref{rem-cost-zero}, there can be ``accidental vanishing'' of those constants. Note also that, by standard linear dispersive estimates
 \[
\vert \braket{\varphi_0}{U(t)\varphi_0}\vert\lesssim \frac{\Vert \varphi_0\Vert^2}{4\pi t},\qquad\text{for large}\: t,
\]
so that the leading order of $\Theta$ is $\f{\log t}{t}$.
\end{rem}

As anticipated at the beginning of the paper, Theorem \ref{teo:cion} states the complete ionization, that is the vanishing of the survival probability $|\Theta(t)|^2$ of the bound state as $t\to+\infty$. This is consistent with what occurs in dimensions one and three. However, in this case it is more surprising since, at any fixed $t$, $\ham$ possesses an eigenvalue independently of the sign of $\alpha(t)$. As a consequence, one could expect, in principle, that the oscillating behavior of $\al(t)$ does not affect the survival probability. However, this is not the case, even though the decay rate \eqref{eq:decay} is slower than in the odd dimensional cases, for which it is estimated by $t^{-3}$. On the other hand, the decay rate is surprising also for the presence of the logarithmic correction, which is missing in odd dimensions and which seems to be a singular feature of the even dimension.

\begin{rem}
 Recall that in 1D the decay rate is actually exact, and the coefficient of the leading order term can be computed for generic smooth initial conditions.
\end{rem}

In addition, we mention that Theorem \ref{teo:cion} does not complete the discussion of the ionization problem as our result holds only for monochromatic perturbations and when the \emph{no-resonance} assumption \eqref{eq:nonres} is satisfied, for technical reasons.

The former is connected to the analysis of the Laplace transform of the charge on the open left half-plane of $\C$ achieved in Section \ref{subsec:lefthalfplane} (see also Remark \ref{rem-monocrom}). This is one of the main technical points of the paper and at the moment it is not clear how to rigorously address it without such requirement, even in odd dimensions (see, e.g., \cite{CCL}). The latter is again connected to the analysis of the Laplace transform of the charge, not only on the open left half-plane (see Remark \ref{rem-whynonres_bis}), the but also on the imaginary axis (see in Section \ref{subsec:inversion}). More precisely, the central point is the analysis at the origin (see Section \ref{subsubsec:pzero}). This problem has been addressed in odd dimensional cases by means of a regularizing technique which is not clear how to adapt the the two-dimensional case (see Remark \ref{rem-whynonres}). Fixing both these technical issue will be addressed in a forthcoming paper.


\subsection{Organization of the paper}

The paper is organized as follows:

\begin{itemize}
 \item[(i)] Section \ref{sec:global} addresses global well-posedness of \eqref{eq:weak_LSE}; more precisely:
 \begin{itemize}
  \item in Section \ref{subsec:charge} we establish existence and uniqueness of the charge in $C([0,T])\cap H^{1/2}(0,T)$ (i.e., (Propositions \ref{prop:charge_cont} and \ref{pro:sobregq}),
  \item absolute continuity of the charge (i.e. Proposition \ref{pro:reg_abs}) is contained in Section \ref{subsec:abscont},
  \item in Section \ref{subsec:wp} we prove Theorem \ref{teo:wp}, relying of the tools developed in Sections \ref{subsec:charge} and \ref{subsec:abscont};
 \end{itemize}
 \item[(ii)] Section \ref{sec:ion} discusses the features of the (unilateral) Laplace transform of the charge; more precisely:
 \begin{itemize}
  \item in Section \ref{subsec:integr} we prove Laplace transformability of the charge;
  \item in Section \ref{subsec:Laplace} we study the behavior of the Laplace transform of the charge on the open right half-plane;
  \item in Section \ref{subsec:inversion} we deal with the behavior of the Laplace transform of the charge on the imaginary axis;
  \item in Section \ref{subsec:lefthalfplane} we conclude the analysis of the Laplace transform of the charge on the open left half-plane;
 \end{itemize}
 \item[(iii)] Section \ref{sec:ionization} contains the proof of Theorem \ref{teo:cion}, relying of the tools developed in Section \ref{sec:ion}.
\end{itemize}

\smallskip

\subsection*{Acknowledgements.}

The authors wishes to thank the anonymous referees for careful reading the manuscript and for their constructive comments that helped improving our results.

\subsection*{Data availability statement}
 Data sharing not applicable to this article as no datasets were generated or analysed during the current study.

\section{Global well-posedness: Proof of Theorem \ref{teo:wp}}
\label{sec:global}

The proof of Theorem \ref{teo:wp} retraces the strategy introduced in \cite{CCT} to study the nonlinear counterpart of the problem, where $\alpha(\cdot)$ is replaced by $|q|^{2\sigma}$. Roughly speaking, we establish an existence and uniqueness result for \eqref{eq:charge_eq} and then prove that the resulting solution $q(\cdot)$ is sufficiently regular so that the function defined by \eqref{eq:ansatz} is the global solution of \eqref{eq:weak_LSE}. Here we focus mainly on those points which require major modifications and new ideas, referring to \cite{CCT} for the already developed tools. Nevertheless, for the sake of completeness, we briefly explain also the parts of the proofs similar to \cite{CCT}.

\smallskip
The actual proof of Theorem \ref{teo:wp} is showed in Section \ref{subsec:wp}, while Sections \ref{subsec:charge} and \ref{subsec:abscont} are devoted to some preliminary results. More precisely, Section \ref{subsec:charge} addresses the $C\cap H^{1/2}$-regularity, whereas Section \ref{subsec:abscont}  deals with $W^{1,1}$-regularity. Note that the proof of the Sobolev regularity of the charge $q$ is the fundamental ingredient of the proof of Theorem \ref{teo:wp}. Indeed, the main point of such proof is the regularity of $\psi(t,\cdot)$, which is in fact reconstructed by a regularity transfer from of the charge $q$. In particular, the key tool is the proof of the $H^{1/2}$-regularity, while the $W^{1,1}$-regularity is in some sense an additional property, which however strongly simplifies the proof, as explained in Remark \ref{rem:regsob2}.

\smallskip
Finally, we recall that throughout the section we will always consider the assumptions of Theorem \ref{teo:wp}, i.e.
\[
 \alpha\in W_{loc}^{1,\infty}\big([0,+\infty)\big)\qquad\text{and}\qquad\psi_0\in\dom.
\]
They are required for the proof of the Theorem, but might be relaxed on every single preliminary step. We will discuss these possible relaxations at the end of all the subsection below (Remarks \ref{rem:dato}, \ref{rem:alfauno}, \ref{rem:regsob}, \ref{rem:alfadue}, \ref{rem:regsob2} and \ref{rem:alfatre}).


\subsection{Existence and uniqueness of the charge} 
\label{subsec:charge}

Here we prove that \eqref{eq:charge_eq} admits a unique solution in $C([0,T])\cap H^{1/2}(0,T)$, for every $T>0$. We establish this result in the next two propositions.

Preliminarily, we introduce the operator $g\mapsto Ig$ such that
\[
 (Ig)(t):=\int_0^t\I(t-\tau)g(\tau)\dtau,
\]
which is nothing but the finite time convolution with the Volterra function of order $-1$. We also mention that we use throughout the symbol $\,\widetilde{\cdot}\,$ to denote the unitary Fourier transform of $\R^2$.

\begin{pro}[Continuous solution of the charge equation]
\label{prop:charge_cont}
 Equation \eqref{eq:charge_eq} admits a unique solution in $C([0,\infty))$.
\end{pro}

\begin{proof}
We start by rewriting \eqref{eq:charge_eq} in the compact form
 \begin{equation}
  \label{eq:charge_comp}
  (\ID-A)q=f,
 \end{equation}
 where
 \begin{equation}
  \label{eq:A_comp}
  Aq:=-I(\zeta q),
 \end{equation}
 with
 \begin{equation}
 \label{eq:zetagreca}
 \zeta(t):=4\pi\left(\alpha(t)+\theta_1-\tfrac{\imath}{8}\right),
 \end{equation}
 and
 \begin{equation}
  \label{eq:for_comp}
  f:=I(4\pi(U(\cdot)\psi_0)(\zev)).
 \end{equation}
 Then, in order to conclude, it is sufficient to exploit \cite[Theorem 5.1.2]{GV}; namely, one has to prove that for any fixed $T>0$:
 \begin{itemize}
  \item[(i)] $f\in C([0,T])$;
  \item[(ii)] $A$ is bounded in $C([0,T])$;
  \item[(iii)] $\displaystyle \sum_{i=0}^\infty A^ng$ converges in $C([0,T])$ for any $g\in C([0,T])$.
 \end{itemize}
 
 Let us begin with item (i). As $ \psi_0=\phi_{\la,0}+q_0\gl\in\dom\subset \domf$,
 \begin{equation}
  \label{eq:Aunoduetre}
  4\pi(U(\tau)\psi_0)(\zev)=4\pi \left(U(\tau)\phi_{\lambda,0} \right)(\zev) + 4\pi q_0 \left(U(\tau)\gl\right)(\zev)=:B_1(\tau)+B_2(\tau).
 \end{equation}
 On the one hand, since
 \[
  B_1(\tau) = 2 \int_{\R^2} e^{-\imath |\kv|^2\tau} \phitr_{\lambda,0}(\kv)\dkv
 \]
 and $\phitr_{\lambda,0}\in L^1(\R^2)$, one sees that $B_1\in L^\infty(0,T)$ and thus $IB_1\in C([0,T])$ by \cite[Lemma 2.3]{CCT}. On the other hand, from \cite[Eq. (2.42)]{CCT},

 \begin{equation}
  \label{eq:Bdue}
  B_2(\tau) =  q_0\left(e^{\imath\lambda\tau}(-\gamma-\log\tau)+\frac{e^{\imath\lambda\tau}}{\pi}\left(-\pi\log\lambda+Q(\lambda,\tau)\right)\right)=:q_0(B_{2,1}(\tau)+B_{2,2}(\tau)),
 \end{equation}
 where $Q(\lambda,\tau)$ is a smooth function of $\tau$ (defined by \cite[Eq. $(2.34)$]{CCT}). Then, again by \cite[Lemma 2.3]{CCT}, $IB_{2,2}\in C([0,T])$. Furthermore, using \cite[Eq. $(2.29)$]{CCT}, there results
 \begin{equation}
  \label{eq:est_Bdue}
  (IB_{2,1})(t)=1+\int_0^t\I(t-\tau) b_{2,1}(\tau)\dtau, \qquad\text{with}\qquad b_{2,1}(\tau) := (e^{\imath\lambda\tau}-1 )(-\gamma-\log\tau),
 \end{equation}
 and thus $IB_{2,1}$ is continuous as before by the boundedness of $b_{2,1}$. Summing up, $IB_2\in C([0,T])$, whence $f\in C([0,T])$.
 
 Concerning item (ii), it suffices to use \eqref{eq:gen_volt} and \cite[Eqs. $(8)$, $(11)$ and $(12)$, Lemma 2.1, Remark 2.1]{CF} to obtain
 \begin{equation}
  \label{ea:A_linear}
  |(Ag)(t)|\leq \zeta_T\,\mu(T,0,0)\,\|g\|_{C[0,T]}, \qquad\text{with}\quad \zeta_T:=\max_{[0,T]}\big|\zeta\big|,
 \end{equation}
 which proves the claim.
 
 It is then left to show (iii). To this aim we first prove by induction that, for every $n\geq 1$,
 \begin{equation}
  \label{eq:ind_rel}
  |(A^ng)(t)|\leq \|g\|_{C[0,T]}\,(\zeta_T)^n\,\mu(t,n-1,0).
 \end{equation}
 The case $n=1$ is given by \eqref{ea:A_linear}. Therefore, assume that \eqref{eq:ind_rel} is satisfied for $n=m$. Consequently,
 \begin{align*}
  |(A^{m+1}g)(t)| & \leq \zeta_T\int_0^t\mu(t-\tau,0,-1)|(A^mg(\tau))|\dtau\\
                  & \leq \|g\|_{C[0,T]}\,(\zeta_T)^{m+1}\int_0^t\mu(t-\tau,0,-1)\mu(\tau,m-1,0)\dtau
 \end{align*}
 and, combining \cite[Eq. $(11)$]{CF} with \cite[Eq. $(69)$]{A} and \cite[Eq. $(13)$]{E}, there results
 \begin{align*}
  \int_0^t\mu(t-\tau,0,-1)\mu(\tau,m-1,0)\dtau & = \frac{\mathrm{d}}{\mathrm{d}t}\int_0^t\mu(t-\tau,0,0)\mu(\tau,m-1,0)\dtau\\
                                                   & = \frac{\mathrm{d}}{\mathrm{d}t}\mu(t,m,1)=\mu(t,m,0),               
 \end{align*}
 which entails \eqref{eq:ind_rel} for $n=m+1$. Now, in view of \eqref{eq:ind_rel}, it is sufficient to show that $\sum_{n=1}^\infty(\zeta_T)^n\mu(T,n-1,0)<+\infty$, or equivalently that
 \begin{equation}
  \label{eq:int_conv}
  \int_0^{+\infty}(\zeta_T)^s\,\mu(T,s,0)\ds<+\infty.
 \end{equation}
 By the Fubini theorem and \eqref{eq:gen_volt} one finds that
 \begin{equation}
 \label{eq:fubini}
  \int_0^{+\infty}(\zeta_T)^s\,\mu(T,s,0)\ds=\int_0^{+\infty}\frac{T^s\,\mu(\zeta_Ts,0,0)}{\Gamma(s+1)}\ds.
 \end{equation}
 Note that the integrating function is continuous on $[0,\infty)$, so that it is sufficient to discuss the behavior for $s\to+\infty$.
 From \cite[Sect. 18.3, pag. 221]{E} and \cite[Eq. 8.327.1$^*$]{GR} we get
 \[
  \mu\big((\zeta_T)s,0,0\big)\sim e^{\zeta_Ts},\qquad\text{as}\quad s\to+\infty,
 \]
 and
 \[
  \Gamma(s+1)\sim\frac{(s+1)^{s+1/2}}{e^{s+1}},\qquad\text{as}\quad s\to+\infty.
 \]
 Consequently,
 \[
  \frac{T^s\,\mu(\zeta_Ts,0,0)}{\Gamma(s+1)}\sim \frac{(Te^{\zeta_T+1})^s}{(s+1)^{s+1/2}}=o(e^{-cs\log s}),\qquad\text{as}\quad s\to+\infty,
 \]
 which proves \eqref{eq:int_conv} via \eqref{eq:fubini}. 
\end{proof}

\begin{rem}
 The previous proof adapts a method introduced by \cite{SY} for the tri-dimensional case. However, in that case, since the integral kernel of the charge equation is the $1/2$-Abel kernel, global existence and uniqueness are straightforward consequences of \cite[Theorem 7.2.1]{GV}. Note also that the proof of the regularity of the forcing term retraces \cite[Proposition 2.2]{CCT}.
\end{rem}

\begin{pro}[$H^{1/2}$-regularity of the charge]
 \label{pro:sobregq}
The solution of \eqref{eq:charge_eq} provided by Proposition \ref{prop:charge_cont} belongs to $H^{1/2}(0,T)$, for every $T>0$.
\end{pro}

In the proof of Proposition \ref{pro:sobregq} we need to exploit the following result, whose proof is technical and completely analogous to \cite[Lemma 2.6]{CCT}. For this reason we omit it for the sake of brevity.

\begin{lem}
 \label{lemma:qlogcont}
 Let $\beta>0$ and define the space of $\beta$-\emph{log-H\"older} continuous functions as
 \begin{multline*}
 {C_{\log,\beta}([0,T]): = \Big\{ g \in C([0,T]): \exists C >0 \text{ such that } \forall t\in[0,T], \: \exists\delta>0 \text{ such that }} \\
 {\forall s \in (t - \delta, t + \delta) \cap [0,T], \: \left| g(t) - g(s) \right| \leq C \left|\log|t - s|\right|^{-\beta} \Big\}}.
\end{multline*}
 If $\alpha\in W_{loc}^{1,\infty}([0,+\infty))$ and $\psi_0\in\dom$, then the solution $q$ of \eqref{eq:charge_eq} provided by Proposition \ref{prop:charge_cont} satisfies
 \[
  q \in \clog([0,T]),\qquad\forall 0<\beta \leq 1,\quad\forall T>0.
 \]
\end{lem}

\begin{proof}[Proof of Proposition \ref{pro:sobregq}]
 The strategy consists of two steps: we first prove the $H^{1/2}$-regularity of the charge $q$ on small intervals and then extend it to arbitrary sub-intervals of $[0,+\infty)$.
 
 \emph{Step (i): there exists $T>0$ such that $q\in H^{1/2}(0,T)$.} The first point is the $H^{1/2}$-regularity of the function $f$ defined by \eqref{eq:for_comp}, for which we rely again on \eqref{eq:Aunoduetre}. Arguing as in \cite[Proof of Proposition 2.3]{CCT}, one can get that
 \begin{equation}
  \label{eq:est_Buno}
  \left\| B_1 \right\|_{H^{\nu}(\R)}^2 \leq C \int_{\R^2} \left(1+|\kv|^4\right)^{\nu}|\phitr_{\lambda,0}(\kv)|^{2}\dkv,
 \end{equation}
 and thus $B_1\in H^{1/2}(0,T)$, for every $T>0$, as $\phi_{\lambda,0}\in H^1(\R^2)$. Since $B_1$ is also bounded we have by \cite[Lemma 2.4]{CCT} that $IB_1\in H^{1/2}(0,T)$, for every $T>0$. On the other hand, since $B_{2,2}$ is smooth, $IB_{2,2}\in H^{1/2}(0,T)$, for every $T>0$ as well and, as $IB_{2,1}=1+I b_{2,1}$ with $b_{2,1}\in H^1(0,T)$, then $IB_{2,1}\in H^{1/2}(0,T)$. Summing up, we obtain that $f\in H^{1/2}(0,T)$.
 
 Now, it is left to show that
 \begin{equation}
  \label{eq:mapG}
  \mathcal{G}(q):=f+Aq,
 \end{equation}
 with $A$ defined by \eqref{eq:A_comp}, is a contraction in a suitable subset of $C[0,T]\cap H^{1/2}(0,T)$, for a sufficiently small $T>0$. Indeed, this entails that \eqref{eq:charge_eq} has a unique solution here and, since it must coincide with that provided by Proposition \ref{prop:charge_cont}, whence the claim is proved. First, define the contraction space as
 \[
 \A_{T}: = \left\{q\in C([0,T])\cap H^{1/2}(0,T):\left\|q\right\|_{L^\infty(0,T)}+\left\|q\right\|_{H^{1/2}(0,T)}\leq \rho_T \right\},
 \]
 with $\rho_T:=2\max\{\|f\|_{L^\infty(0,T)}+\|f\|_{H^{1/2}(0,T)},1\}$. The set $ \A_{T}$ is a complete metric space with the norm induced by $C[0,T]\cap H^{1/2}(0,T)$, i.e.,
 \[
  \left\| \cdot \right\|_{\A_{T}}:= \left\|\cdot \right\|_{L^\infty(0,T)}+ \left\|\cdot \right\|_{H^{1/2}(0,T)}.
 \]
 Let us begin with showing that $ \mathcal{G} $ maps $ \A_T $  into  itself. If $q\in \A_{T}$, then one easily sees that $\mathcal{G}(q)$ is continuous. On the other hand, from \cite[Lemma 2.4.]{CCT}, one obtains
 \[
  \left\|Aq\right\|_{H^{1/2}(0,T)} \leq C\left\|I(\zeta q)\right\|_{\A_{T}}\leq C_T \left\|\zeta q \right\|_{\A_{T}} \leq C_T \zeta_{\A_T} \left\| q \right\|_{\A_T} \leq C_T  \zeta_{\A_T} \rho_T
 \]
 where $\zeta_{\A_T}:=\|\zeta\|_{\mathcal{A}_T}$ and, throughout, $C_T$ is a generic positive constant such that $C_T \to 0$, as $ T \to 0 $ (and which may vary from line to line). In addition, by \cite[Lemma 2.3]{CCT}, one sees that
 \begin{equation}
  \label{eq:estGuno_inf}
  \left\|Aq\right\|_{L^\infty(0,T)} \leq C \left\|I(\zeta q)\right\|_{L^\infty(0,T)} \leq C_T  \left\|\zeta q \right\|_{L^\infty(0,T)} \leq C_T\zeta_{\A_T}\rho_T,
 \end{equation}
 so that
 \[
  \left\|Aq\right\|_{\A_{T}}\leq C_T\zeta_{\A_T}\rho_T.
 \]
 Thus, we have
\[
  \left\|\mathcal{G}(q) \right\|_{\A_{T}} \leq \rho_T \left(  \tfrac{1}{2}+ C_T \zeta_{\A_T}\right)
 \]
 and, as $\zeta_{\A_T}$ is bounded, the term in brackets is equal to $\frac{1}{2}+o(1)$ as $T\to0$, so that $\mathcal{G}(q)\in \A_{T}$, for $T$ sufficiently small. Therefore, as a final step, we show that $\mathcal{G}$ is actually a norm contraction. For any pair of functions $q_1,\,q_2\in \A_{T}$, we have
 \begin{equation}
  \label{eq:relG}
  \mathcal{G}(q_1)-\mathcal{G}(q_2)=A(q_2-q_1)
 \end{equation}
 and hence, arguing as before,
 \[
  \left\|\mathcal{G}(q_1)-\mathcal{G}(q_2)\right\|_{\A_{T}}\leq C_T\zeta_{\A_T}\left\| q_1-q_2 \right\|_{\A_T},
\]
whence $\mathcal{G}$ is a contraction on $\A_{T}$, since again $C_T\to0$, as $T\to0$, and $\zeta_{\A_T}$ is bounded.
 
 \emph{Step (ii): $q\in H^{1/2}(0,T)$ for every $T>0$.} By Step (i), there exists $T_0>0$ such that the solution $q$ of \eqref{eq:charge_eq} belongs to $H^{1/2}(0,T_0)$. Now, consider the equation
 \begin{equation}
  \label{eq:charge_new}
  q_1(t) + \int_0^t \dtau \:\I(t-\tau)\zeta(T_0+\tau) q_1(\tau) = f_1(t),
 \end{equation}
 where
 \begin{equation}
  \label{eq:new_f}
  f_1(t):=f(t+T_0)-\int_0^{T_0}\dtau\:\I(t+T_0-\tau)\zeta(\tau)q(\tau).
 \end{equation}
 Combining the regularity properties of $f$ established in the proof of Proposition \ref{prop:charge_cont} and in Step (i) we have that $f(\cdot+T_0)\in C[0,T]\cap H^{1/2}(0,T)$, for every $T>0$. On the other hand, exploiting \cite[Lemma 2.7]{CCT} with $T=T_0$ and $h=\zeta q$, there results that also the second term of \eqref{eq:new_f} belongs to $C[0,T]\cap H^{1/2}(0,T)$,  for every $T>0$, and thus $f_1\in C[0,T]\cap H^{1/2}(0,T)$, for every $T>0$. Consequently, arguing as before, there exist $T_1>0$ and $q_1\in C[0,T_1]\cap H^{1/2}(0,T_1)$ which solves \eqref{eq:charge_new}. In addition, since $q(t)=q_1(t-T_0)$ for every $t\in[T_0,T_0+T_1]$, one sees that $q\in H^{1/2}(0,T_0)$ and $q\in H^{1/2}(T_0,T_0+T_1)$. In principle, this is not sufficient to claim that $q\in H^{1/2}(0,T_0+T_1)$ due to the non-locality of the $H^{1/2}$-norm and the failure of the Hardy inequality (as explained in \cite[Section 2.1]{CCT}). However, thanks to Lemma \ref{lemma:qlogcont}, we know a priori that $q\in C_{\log,1}[0,T]$ for all $T>0$, so that one can argue as in \cite[Proofs of Propositions 2.1$\&$2.4]{CCT} and obtain nevertheless that $q\in H^{1/2}(0,T_0+T_1)$.
 
 Summing up, we showed that, if one proves $H^{1/2}$-regularity of $q$ up to a time $ T_0>0$, then it can be extended up to $ T_0 + T_1 $ for some $T_1>0$. This argument allows in fact to extend the $H^{1/2}$-regularity of the charge up to any time $T>0$. To this aim fix an arbitrary $T>0$. If there exists $j^*$ such that $\sum_{j=0}^{j^*}T_j\geq T$, where each $T_j$ is the resulting time of the $j$-th iteration of the previous procedure, then the claim is proved.
 
 By Step (i), each $T_j$ has to be chosen so that
 \[
  C_{T_j}\zeta_{\A_{\sum_{\ell=0}^{j-1}T_\ell}}<\frac{1}{2}.
 \]
 However, one can chose $C_{T_j}$ in such a way that $C_{T_j}>c$ for some constant $c>0$ (as it actually suffices that $C_{T_j}<1/2\zeta_{\A_T}$). As a consequence, one can chose also $T_j$ in such a way that $T_j>c'$ for some constant $c'>0$, whence $\sum_{j=0}^\infty T_j=+\infty$.
\end{proof}

 \begin{rem}
  \label{rem:dato}
  One can see that Propositions \ref{prop:charge_cont} and \ref{pro:sobregq} do not exploit the whole regularity of $\phi_{\la,0}$ provided by \eqref{eq:extradom}. As mentioned in Section \ref{subsec:main}, a weaker assumption is sufficient; namely, one may only assume $\psi$ to belong to
  \begin{equation}
 \label{eq:extradom2}
  \dom_1:= \left\{\psi\in\domf:\phitr_\lambda({\bf k}) \in L^1(\R^2)\right\},
 \end{equation}
 to get Propositions \ref{prop:charge_cont} and \ref{pro:sobregq}, and thus Theorem \ref{teo:wp}. However, it is not possible at the moment to further weaken the assumptions on the initial datum. For instance, the natural assumption $\psi_0\in\domf$ does not work as we cannot prove the regularity of the forcing term $f$ of \eqref{eq:charge_comp} without requiring that $\phitr_{\la,0}\in L^1(\R^2)$. This fact occurs also in the nonlinear point interaction problem (see \cite[Remark 1.5]{CCT}). Precisely, in the nonlinear case \eqref{eq:extradom2} is replaced by a slightly stronger requirement (see \cite[Eq. $(1.17)$]{CCT}), but this is connected to the proof of the energy conservation which is neither required nor possible in the linear non-autonomous case considered in the present paper. 
 \end{rem}
 
 \begin{rem}
  \label{rem:alfauno}
  Concerning $\alpha$, one can check that Proposition \ref{prop:charge_cont} can be proved under the sole assumption $\alpha\in C([0,T])$, for every $T>0$, whereas Proposition \ref{pro:sobregq} can be proved under the sole assumption $\alpha\in C([0,T])\cap H^{1/2}(0,T)$, for every $T>0$.
 \end{rem}


\subsection{Further regularity of the charge}
\label{subsec:abscont}

Here we prove some further Sobolev regularity of the solution of \eqref{eq:charge_eq} obtained in Section \ref{subsec:charge}. The strategy to prove such a further regularity has been developed in \cite[Section 5.1]{ACCT}.

\begin{pro}[$W^{1,1}$-regularity of the charge]
 \label{pro:reg_abs}
The solution of \eqref{eq:charge_eq} provided by Proposition \ref{prop:charge_cont} belongs to $W^{1,1}(0,T)$, for every $T>0$.
\end{pro}

\begin{proof}
 As for Proposition \ref{pro:sobregq}, the proof consists of two steps: the $W^{1,1}$- regularity of the charge $q$ on small intervals and the extension of such regularity to any sub-interval of $[0,+\infty)$.

 \emph{Step (i): there exists $T>0$ such that $q\in W^{1,1}(0,T)$.} Exploiting again the compact form of \eqref{eq:charge_eq} provided by \eqref{eq:charge_comp}, the first point is to prove the $W^{1,1}$-regularity of $f$. We rely again on the decomposition given by \eqref{eq:Aunoduetre}. Combining, \eqref{eq:est_Buno} with \eqref{eq:extradom}, we see that $B_1\in H^1(0,T)$, and hence $B_1\in W^{1,1}(0,T)$, for every $T>0$. Consequently, by \cite[Theorem 5.3]{CF}, $IB_1\in W^{1,1}(0,T)$. On the other hand, recalling \eqref{eq:Bdue} and \eqref{eq:est_Bdue} and using again \cite[Theorem 5.3]{CF}, $IB_2\in W^{1,1}(0,T)$, and thus $f\in W^{1,1}(0,T)$. Notice that \cite[Theorem 5.3]{CF} is actually proved for real-valued functions, but a direct inspection of the proof shows that it holds for complex-valued functions as well .
 
 Now, it is left to show that the map $\mathcal{G}$ defined by \eqref{eq:mapG} is a contraction in a suitable subset of $W^{1,1}(0,T)$, for a sufficiently small $T>0$. As contraction space we choose
 \[
  \B_{T} := \left\{q\in W^{1,1}(0,T): \left\|q\right\|_{L^\infty(0,T)}+\left\|q\right\|_{W^{1,1}(0,T)}\leq \eta_{T} \right\},
 \]
 with $\eta_{T}:=2\max\{\|f\|_{L^\infty(0,T)}+\|f\|_{W^{1,1}(0,T)},1\}$, which is a complete metric space whenever endowed with the norm
 \[
  \left\| \cdot \right\|_{\B_{T}}:= \left\|\cdot \right\|_{L^\infty(0,T)}+ \left\|\cdot \right\|_{W^{1,1}(0,T)}.
 \]
 First, we show that $\mathcal{G}(\B_{T})\subset \B_{T}$. From \cite[Theorem 5.3]{CF}, setting $\zeta_{\B_T}:=\|\zeta\|_{\B_T}$,
 \[
  \left\|Aq\right\|_{W^{1,1}(0,T)} \leq C_T \left\| \zeta q \right\|_{\B_{T}} \leq C_T \zeta_{\B_T} \left\| q \right\|_{\B_T} \leq C_T  \zeta_{\B_T}\eta_T,
 \]
 where $C_T$ denotes a generic positive constant such that $C_T\to0$, as $T\to0$, and that can be redefined from line to line. Hence, combining with \eqref{eq:estGuno_inf},
 \begin{equation}
  \label{eq:stimaguno_bis}
  \left\|Aq\right\|_{\B_{T}}\leq C_T\zeta_{\B_T}\eta_{T},
 \end{equation}
 so that
 \[
  \left\|\mathcal{G}(q) \right\|_{\B_{T}} \leq\eta_{T} \left(  \tfrac{1}{2}+ C_T\zeta_{\B_T}\right).
 \]
 Consequently, as the term in brackets is equal to $\frac{1}{2}+o(1)$ as $T\to0$, $\mathcal{G}(q)\in \B_{T}$ for $T$ sufficiently small. Finally, consider two functions $q_1,\,q_2\in \B_{T}$. By \eqref{eq:relG} and \eqref{eq:stimaguno_bis}
 \[
  \left\|\mathcal{G}(q_1)-\mathcal{G}(q_2) \right\|_{\B_{T}} \leq C_T \zeta_{\B_T} \left\|q_1-q_2 \right\|_{\B_{T}}
 \]
 and again, since $C_T\to0$ as $T\to0$, $\mathcal{G}$ is a contraction on $\B_{T}$ for $T$ sufficiently small.
 
  \emph{Step (ii): $q\in W^{1,1}(0,T)$ for every $T>0$.} By Step (i), there exists $T_0>0$ such that $q\in W^{1,1}(0,T_0)$. Consider, then, as in the proof of Proposition \ref{pro:sobregq}, equation \eqref{eq:charge_new}. Arguing as before and using \cite[Lemma 5.2]{ACCT} with $T=T_0$ and $h=\zeta q$, one can see that $f_1\in W^{1,1}(0,T)$ for every $T>0$. Hence, there exist $T_1>0$ and $q_1\in W^{1,1}(0,T_1)$ which solves \eqref{eq:charge_new}. Since $q(t)=q_1(t-T_0)$ for every $t\in[T_0,T_0+T_1]$, one finds that $q\in W^{1,1}(0,T_0+T_1)$.
  
  In other words, if one establishes $W^{1,1}$-regularity up to a time $T_0$, then it can be extended up to $T_0+T_1$ for some $T_1>0$. However, as for the $H^{1/2}$-regularity, one can easily check that such a procedure allows in fact to extend $W^{1,1}$-regularity up to any time $T>0$.
\end{proof}

\begin{rem}
  \label{rem:regsob}
  In contrast to Propositions \ref{prop:charge_cont} and \ref{pro:sobregq}, Proposition \ref{pro:reg_abs} requires that $\psi_0\in\dom$. Indeed, if one assumes only $\psi_0\in\dom_1$, then one cannot guarantee in general that $B_1\in H^1(0,T)$ and, thus, that $B_1\in W^{1,1}(0,T)$ and $I B_1\in W^{1,1}(0,T)$.
 \end{rem}
 
 \begin{rem}
  \label{rem:alfadue}
  Concerning $\alpha$, one can check that Proposition \ref{pro:reg_abs} can be proved under the sole assumption $\alpha\in W^{1,1}([0,T])$, for every $T>0$.
 \end{rem}


\subsection{Global well-posedness of (\ref{eq:weak_LSE})}
\label{subsec:wp}

Exploiting the tools developed in the previous sections we can finally give the proof of Theorem \ref{teo:wp}.

\begin{proof}[Proof of Theorem \ref{teo:wp}]
 We have to prove that, as one plugs the charge $q$ obtained by Propositions \ref{prop:charge_cont}, \ref{pro:sobregq} and \ref{pro:reg_abs} into \eqref{eq:ansatz}, the resulting function $\psi$ is the unique solution of \eqref{eq:weak_LSE} and satisfies the mass conservation.

 Concerning uniqueness one can argue as in the nonlinear case, since it is a direct consequence of the uniqueness of the solution of the charge equation and of the ansatz \eqref{eq:ansatz} (for more details see \cite[Remark 1.3]{CCT}). On the other hand, if one is able to show that $ \psi(t,\cdot) \in \domf $ for all $t>0$, then the fact that \eqref{eq:ansatz} solves \eqref{eq:weak_LSE}, whenever $ q(t) $ solves \eqref{eq:charge_eq}, can be proved as in \cite[Theorem 2.1]{DFT2} (see also \cite{CCF}). In that case the model is different (\emph{moving point interactions} are discussed), but the argument is completely analogous and we omit it for the sake of brevity. Finally, the mass conservation can be proved repeating the same computations of \cite[Proof of Theorem 1.2 - Part 1]{CCT} (it is actually easier thanks to the further regularity of the charge).
 
Therefore, it is left to prove that $\psi(t,\cdot) \in \domf $ for any $t>0$, that is
 \[
  \psi(t,\cdot) - q(t)\gl\in H^1(\R^2).
 \]
 Using the Fourier transform and an integration by parts (as in \cite[Proof of Theorem 1.1]{CCT}), this is equivalent to show that
 \begin{equation}
  \label{eq:regparts}
  e^{-\imath|\kv|^2t}\left(\psitr_0(\kv)-\frac{q(0)}{2\pi(|\kv|^2+\lambda)}\right)-\frac{1}{2\pi(|\kv|^2+\lambda)}\int_0^te^{-\imath|\kv|^2(t-\tau)}(\dot{q}(\tau)-\imath\lambda q(\tau))\dtau
 \end{equation}
 belongs to $L^2(\R^2,(|\kv|^2+1)\dkv)$. As the former term is the Fourier transform of $U(t)\phi_{\lambda,0}$, it clearly belongs to $L^2(\R^2,(|\kv|^2+1)\dkv)$ (actually $U(t)\phi_{\lambda,0}\in H^2(\R^2)$). Concerning the latter term, we first set $ \lambda = 1 $ (for the sake of simplicity) and then change variables to get
 \begin{multline*}
  \displaystyle \int_{\R^2}\bigg|\frac{1}{2\pi(|\kv|^2+1)}\int_0^te^{-\imath|\kv|^2(t-\tau)}(\dot{q}(\tau)-\imath q(\tau))\dtau\bigg|^2(1+|\kv|^2)\dkv\\[.3cm]
  \displaystyle \leq C\int_0^{+\infty}\bigg[\bigg|\int_0^te^{\imath\varrho\tau}\dot{q}(\tau)\dtau\bigg|^2+\bigg|\int_0^te^{\imath\varrho\tau}q(\tau)\dtau\bigg|^2\bigg]\frac{\drho}{1+\varrho}.
 \end{multline*}
 Furthemore,
 \[
  \int_0^te^{\imath\varrho\tau}\dot{q}(\tau)\dtau = \sqrt{2\pi}\:\widetilde{{\dot \xi}} (-\varrho),\qquad\text{and}\qquad\int_0^te^{\imath\varrho\tau}q(\tau)\dtau = \sqrt{2\pi} \: \widetilde{\one_{[0,t]}q}(-\varrho),
 \]
 with
 \[
  \xi(\tau):=
  \begin{cases}
   q(0),    & \mbox{if }\tau\leq0,\\
   q(\tau), & \mbox{if }0<\tau<t,\\
   q(t),    & \mbox{if }\tau\geq t,
  \end{cases}
 \]
 so that
 \[
  \int_0^{+\infty}\bigg[\bigg|\int_0^te^{\imath\varrho\tau}\dot{q}(\tau)\dtau\bigg|^2+\bigg|\int_0^te^{\imath\varrho\tau}q(\tau)\dtau\bigg|^2\bigg]\frac{\drho}{1+\varrho}\leq  2\pi \int_\R\frac{ | \widetilde{{\dot \xi}} (\varrho) |^2}{1+|\varrho|}\drho+\int_\R\frac{|\widetilde{\one_{[0,t]}q}(\varrho)|^2}{1+|\varrho|}\drho.
 \]
 Finally, since $q\in C_{\log,1}[0,T]\cap H^{1/2}(0,T)$, (arguing as in \cite[Proof of Proposition 2.1]{CCT}) one finds that $\xi\in \dot{H}^{1/2}(\R)$ and thus $\dot{\xi}\in H^{-1/2}(\R)$. As a consequence, also the latter term of \eqref{eq:regparts} belongs to $L^2(\R^2,(|\kv|^2+1)\dkv)$, which completes the proof.
\end{proof}

\begin{rem}
  \label{rem:regsob2}
  Comparing the above proof with \cite[Proof of Theorem 1.1]{CCT} one can see the simplifying role of the absolute continuity of the charge. In the previous proof, indeed, the duality arguments developed in \cite[Remarks 2.4 and 2.5]{CCT} are not required in order to justify the integrals involved. However, using those tools one could prove Theorem \ref{teo:wp} as well even without absolute continuity, that is, even with the assumption $\psi_0\in\dom$ replaced by $\psi_0\in\dom_1$.
\end{rem}

\begin{rem}
 \label{rem:dominio}
 In the previous proof one immediately sees that, in order to establish strong well-posedness in the operator domain, one should be able to prove $H^1$-regularity of the charge, possibly adding further assumptions on $\psi_0$. Unfortunately, to the best of our knowledge, this is not the possible due to the very singular behavior of the integral kernel $\I$ (see \cite[Remarks 5.5 and 5.8]{CF}).
\end{rem}

\begin{rem}
  \label{rem:alfatre}
  In the previous proof the assumption $\alpha\in W^{1,\infty}_{loc}([0,+\infty))$ seems to play no actual role. However, it is required in order to prove that \eqref{eq:ansatz} solves \eqref{eq:weak_LSE} whenever $ q $ solves \eqref{eq:charge_eq} arguing as in \cite[Theorem 2.1]{DFT2}.
 \end{rem}


\section{The Laplace transform of the charge}
\label{sec:ion}

Following the strategy introduced by \cite{CDFM,CCLR}, in order to prove Theorem \ref{teo:cion} it is necessary to discuss the regularity properties of the (unilateral) Laplace transform of the charge. Our analysis retraces the method used in \cite{CCL} for the one-dimensional case. In particular, we split the analysis considering first the right open half-plane,  then the imaginary axis and the left open half-plane. However, as a preliminary step, we have to prove that the charge actually admits a Laplace transform.
\smallskip

Note that in the following we tacitly assume that $\psi_0=\varphi_0$, with $\varphi_0$ defined by \eqref{eq:eigen}, and that $\alpha$ satisfies the monochromatic assumption with $\eta=0$ and $c=0$. For a brief outline on the modifications required to deal with the general case we refer the reader to the Appendix \ref{appendix}.

\smallskip

As $\varphi_0\in\dom$ and $\alpha_0\sin(\omega\,\cdot)\in W^{1,\infty}_{loc}([0,\infty))$, Theorem \ref{teo:wp} is valid, as well as Propositions \ref{prop:charge_cont}, \ref{pro:sobregq} and \ref{pro:reg_abs} and Lemma \ref{lemma:qlogcont}. Hence, $\psi$ and the charge $q$ are meant throughout as the solutions of \eqref{eq:weak_LSE} and \eqref{eq:charge_eq}, respectively, with the above choice of $\psi_0$ and $\alpha$.


\subsection{Laplace transformability of the charge}
\label{subsec:integr}

As already mentioned, the preliminary step is to prove the actual Laplace transformability of the charge. Namely, we need to show that there exists $\nu>0$ such that
\[
 \int_0^{+\infty} e^{-\nu t}\,|q(t)|\dt<+\infty.
\]
To this aim, for any fixed $T\in(0,+\infty]$ and $\nu>0$, we define the space
\[
 L_{\nu}^T:=\left\{f:[0,+\infty)\to\C:\|f\|_{\nu,T}:=\int_{0}^{T}e^{-\nu s}\,|f(s)|\ds<+\infty\right\}.
\]
Then, we can prove the following result.

\begin{pro}
 \label{pro:q_lap}
 There exists $\nu>1$ such that the charge $q(t)$ belongs to $L_{\nu}^{+\infty}$.
\end{pro}

In the proof of Proposition \ref{pro:q_lap}, we use again the compact form of the charge equation given by \eqref{eq:charge_comp}. In particular, we start by establishing the following property for the function $f$ defined by \eqref{eq:for_comp}.

\begin{lem}
 \label{lem:forz_exp}
 There exists $\nu>1$ such that the function $f$ defined by \eqref{eq:for_comp} belongs to $L_{\nu}^{+\infty}$.
\end{lem}

\begin{proof}
 First, recall that
 \[
  (U_0(\tau)\varphi_0)(\zev)=\frac{1}{2\pi}\int_{\R^2}e^{-\imath|\kv|^2\tau}\vphitr_{0}(\kv)\dkv=\frac{\const_0}{4\pi^2}\int_{\R^2}\frac{e^{-\imath|\kv|^2\tau}}{|\kv|^2-\lambda_0}\dkv.
 \]
 Consequently, from \cite[Eq. (2.42)]{CCT}, we have
 \[
  (U_0(\tau)\varphi_0)(\zev)=\frac{\const_0}{4\pi}\left(-e^{-\imath\lambda_0\tau}\log(-\lambda_0)+e^{-\imath\lambda_0\tau}(-\gamma-\log\tau)+\frac{e^{-\imath\lambda_0\tau}}{\pi}Q(-\lambda_0,\tau)\right)
 \]
 (with $Q(\cdot\,,\cdot)$ given again by \cite[Eq. (2.34)]{CCT}). Hence,
 \begin{align*}
  f(t) = & \, -\const_0\log(-\lambda_0)\int_0^t\I(t-\tau)e^{-\imath\lambda_0\tau}\dtau+\const_0\int_0^t\I(t-\tau)e^{-\imath\lambda_0\tau}(-\gamma-\log\tau)\dtau+\\[.2cm]
         & \, +\frac{\const_0}{\pi}\int_0^t\I(t-\tau)e^{-\imath\lambda_0\tau}Q(-\lambda_0,\tau)\dtau=:J_1(t)+J_2(t)+J_3(t).
 \end{align*}
 Now, recalling that $\mathcal{I}\geq 0$, one observes that $J_1$ is continuous and satisfies
 \[
  |J_1(t)|\leq\const_0|\log(-\lambda_0)|\NN(t),
 \]
 with
 \begin{equation}
  \label{eq:NN}
  \NN(t):=\int_0^t\I(\tau)\dtau=\mu(t,0,0).
 \end{equation}
 Therefore, since by \cite[Section 18.3, pag 221]{E} $\NN(t)\sim e^t$, as $t\to+\infty$, one sees that $J_1\in L_{\nu}^{+\infty}$, for every $\nu>1$. On the other hand, there results that
 \[
  J_2(t)=\const_0\int_0^t\I(t-\tau)(-\gamma-\log\tau)\dtau+\const_0\int_0^t\I(t-\tau)\underbrace{\big(e^{-\imath\lambda_0\tau}-1)(-\gamma-\log\tau)}_{:=\ell(\tau)}\dtau.
 \]
 From \cite[Eq. (2.29)]{CCT}, the first integral in the above formula is equal to 1 for every $t\geq0$. Hence, combining $\|\ell\|_{L^\infty(0,t)}\leq C(1+\vert\log t\vert)$ with \cite[Proposition 4.2]{CF}, we obtain
 \[
  |J_2(t)|\leq\const_0\big(1+C\NN(t)(1+\vert \log t\vert)\big)
 \]
 and thus, again by \cite[Section 18.3, pag 221]{E}, also $J_2\in L_\nu^{+\infty}$, for every $\nu>1$. Finally, using \cite[Eq. (2.33)]{CCT},
 \[
  \tfrac{1}{\pi}Q(-\lambda_0,\tau)=\gamma+\log(-\lambda_0\tau)-\mathrm{ci}(-\lambda_0\tau)+\imath\mathrm{si}(-\lambda_0\tau),
 \]
 where $\mathrm{si}(\cdot)$ and $\mathrm{ci}(\cdot)$ stand for the \emph{sine} and \emph{cosine integral functions} defined by \cite[Eqs. 8.230]{GR}. Using the definition and \cite[Eqs. 3.721]{GR}, one sees that the last term is bounded on $[0,+\infty)$; whereas, using 
 \[
  \gamma+\log(-\lambda_0\tau)-\mathrm{ci}(-\lambda_0\tau)=-\int_0^{-\lambda_0\tau}\frac{\cos s-1}{s}\ds
 \]
 (again by \cite[Eqs. 8.230]{GR}), one sees that $\gamma+\log(-\lambda_0\tau)-\mathrm{ci}(-\lambda_0\tau)$ is locally bounded on $[0,+\infty)$ and bounded by  $C\log(-\lambda_0\tau)$, for some constant $C>0$, as $\tau \to +\infty$. Hence, arguing as before one finds that $J_3\in L_\nu^{+\infty}$ for every $\nu>1$ as well, thus concluding the proof.
\end{proof}

We are now in a position to establish the Laplace transformability of the charge. To this aim, it is sufficient to prove that the map $\g$ introduced by \eqref{eq:mapG} is a contraction in a suitable space that contains the solution of \eqref{eq:charge_eq} obtained in Section \ref{sec:global}.

\begin{proof}[Proof of Proposition \ref{pro:q_lap}]
 Preliminarily, we fix some parameters. Recalling \eqref{eq:zetagreca}, let $\zeta_\infty:=\max_{[0,+\infty)}\big|\zeta\big|$ and let $T_*>0$ be such that $\I(t)\leq C_{T_*}e^t$ for all $t\geq T_*$ and for some finite $C_{T*}$ (this is possible by \cite[Eq. $(8)$]{CF}). As a consequence, recalling that $\NN(t)\to0$, as $t\to0$, we can also fix $T_1>0,\,T_2>\max\{T_*,T_1\},\text{ and }\nu>1$ such that
 \begin{equation}
  \label{eq:cond_aux}
  \left\{\begin{array}{l}
   \displaystyle \NN(T_1)<\frac{1}{2\zeta_\infty},\\[.5cm]
   \displaystyle \I(T_2)>\I(T_1),\\[.5cm]
   \displaystyle \frac{\I(T_2)}{\nu}+\frac{C_{T_*}}{\nu-1}<\frac{1}{2\zeta_\infty}-\NN(T_1).
  \end{array}\right.
 \end{equation}
  
 Now, let $T>T_2$ and consider the map $\mathcal{G}(q)$ defined by \eqref{eq:mapG}. Using \cite[$\text{(A1)--(A3)}$]{CCLR}, one can see that
 \begin{equation}
 \label{eq:alg_exp}
  \|\mathcal{G}(q)\|_{\nu,T}\leq\|f\|_{\nu,T}+\zeta_\infty\|\I\|_{\nu,T}\|q\|_{\nu,T},\qquad\forall q\in L_{\nu,T}.
 \end{equation}
 On the other hand, using Lemma \ref{lem:forz_exp}, we can define
 \[
  b_{\nu,f}:=2\max\{1,\|f\|_{\nu,\infty}\}\qquad\text{and}\qquad B_{\nu,T,f}:=\{q\in L_{\nu,T}:\|q\|_{\nu,T}\leq b_{\nu,f}\}.
 \]
 Hence, combining with \eqref{eq:alg_exp}, there results
 \[
  \|\mathcal{G}(q)\|_{\nu,T}\leq b_{\nu,f}\left(\frac{1}{2}+\zeta_\infty\|\I\|_{\nu,T}\right),\qquad\forall q\in B_{\nu,T,f}.
 \]
 and thus, since \eqref{eq:cond_aux} yields
 \begin{align*}
  \|\I\|_{\nu,T} & =\int_0^{T_1}\dt\:e^{-\nu t}\I(t)+\int_{T_1}^{T_2}\dt\:e^{-\nu t}\I(t)+\int_{T_2}^{T}\dt\:e^{-\nu t}\I(t)\\[.2cm]
                 & \leq \NN(T_1)+\frac{\I(T_2)}{\nu}+\frac{C_{T_*}}{\nu-1}<\frac{1}{2\zeta_\infty},
 \end{align*}
 we obtain that $\mathcal G(B_{\nu,T,f})\subset B_{\nu,T,f}$. Similarly, \eqref{eq:alg_exp} implies that there exists $L\in(0,1)$ such that
 \[
  \|\mathcal{G}(q_1)-\mathcal{G}(q_2)\|_{\nu,T}\leq L\|q_1-q_2\|_{\nu,T},\qquad\forall q_1,q_2\in B_{\nu,T,f}.
 \]
 Consequently, $\mathcal{G}$ is a contraction on $B_{\nu,T,f}$, so that the unique solution of \eqref{eq:charge_eq} belongs to $L_{\nu,T}$ and satisfies $\|q\|_{\nu,T}\leq b_{\nu,f}$, for every $T>0$. However, since $\nu$ is independent of $T$, this entails that, letting $T\to\infty$, one finds $q\in L_\nu^{+\infty}$.
\end{proof}

\begin{rem}\label{rem:alpha}
Note also that Proposition \ref{pro:q_lap} holds even if one replaces the monochromatic assumption with the weaker requirement that $\alpha\in L^{\infty}([0,+\infty))$, provided that the existence of a unique charge is already known.
\end{rem}


\subsection{Extension of the Laplace transform of the charge to the whole open right half-plane}
\label{subsec:Laplace}

A straightforward byproduct of Proposition \ref{pro:q_lap} is that the Laplace transform of the charge, that we denote throughout by $\widehat{q}$,  is well defined and analytic at least on a set $\{p\in\C:\re(p)>\nu\}$, for some $\nu>1$. Here, we aim at showing that it can be actually extended by analyticity at least to the whole open right half-plane.

\begin{pro}
\label{pro:q_halfplane}
The Laplace transform $\widehat q$ of the charge is well defined and analytic on $\{p\in\C: \re(p)>0\}$.
\end{pro}

\begin{proof}
 Recalling \eqref{eq:survival} and \eqref{eq:ansatz}, one can see that 
 \begin{equation}
 \label{eq:theta}
  \Theta(t)=\braket{\varphi_0}{U(t)\varphi_0}+\\
  +\underbrace{\frac{\imath}{2\pi}\lim_{R\to\infty}\int_0^tq(\tau)\bigg(\int_{B_R}e^{-\imath|\kv|^2(t-\tau)}\,\vphitr_0(\kv)\dkv\bigg)\dtau}_{=:Z(t)},
 \end{equation}
 where $B_R$ is the ball of radius $R$ centered at the origin. As the mass is preserved along the flow and $U$ is unitary on $L^2(\R^2)$, we have that
 \[
  |\Theta(t)|\leq\|\varphi_0\|_{L^2(\R^2)}^2,\qquad|\braket{\varphi_0}{U(t)\varphi_0}|\leq\|\varphi_0\|_{L^2(\R^2)}^2
 \]
 and thus both $\Theta$ and $\braket{\varphi_0}{U(\cdot)\varphi_0}$ admit an analytic Laplace transform on $\{p\in\C: \re(p)>0\}$, whence the same is true for $Z$. 
 
 Then, let us compute the Laplace transform of $Z$. First, combining \eqref{eq:eigen} with \cite[Eq. (2.33)]{CCT} and \cite[Eqs. (8.231)]{GR}, there results that
 \begin{multline*}
  \int_{B_R}e^{-\imath|\kv|^2(t-\tau)}\,\vphitr_0(\kv)\dkv=-\frac{\const_0}{2} e^{-\imath\lambda_0(t-\tau)}\times\\[.3cm]
  \times\left(\rm{ci}\big(-\lambda_0 (t-\tau)\big)-\imath\rm{si}\big(-\lambda_0 (t-\tau)\big)-\rm{ci}\big((R^2-\lambda_0)(t-\tau)\big)+\imath\rm{si}\big((R^2-\lambda_0)(t-\tau)\big)\right).
 \end{multline*}
 Since the former term in brackets at the right-hand side is independent of $R$, and the latter converges pointwise to zero, as $R\to\infty$, and is estimated by $C(1+|\log(t-\tau)|)$ (for details see \cite[Pag. 34]{ACCT}), dominated convergence yields
 \begin{align}
 \label{eq:zeta_impl}
  Z(t)= & \, \imath\int_0^tq(\tau)(U(t-\tau)\varphi_0)(\zev)\dtau\nonumber\\[.3cm]
      = & \, -\frac{\imath\const_0}{4\pi} \int_0^tq(\tau)\,e^{-\imath\lambda_0(t-\tau)}\left(\rm{ci}\big(-\lambda_0 (t-\tau)\big)-\imath\rm{si}\big(-\lambda_0 (t-\tau)\big)\right)\dtau.
 \end{align}
 Now, by \cite[Pag. 1115]{GR} we can easily deduce the expression of the Laplace transform of $\rm{ci}$ and $\rm{si}$, given by
\[
\cL({\rm{ci}})(p)=-\frac{1}{2p}\log(1+p^2)\,,\qquad \cL({\rm{si}})(p)=\frac{1}{p}\arctan(1/p)-\frac{\pi}{2p}\,,\quad \re(p)>0.
\]
Then we can compute the Laplace transform of $Z$ using the finite-time convolution property to get
 \begin{equation}
  \label{eq:zeta_tr}
  \widehat{Z}(p)=\widehat{Z}_{2}(p)\,\widehat{q}(p)
 \end{equation}
 with
 \begin{multline}\label{laplace2}
  \widehat{Z}_{2}(p)=\\
  -\frac{\imath\const_0}{4\pi}\,\left[-\frac{\log\left(1+(p+ \imath\lambda_0)^2/(-\lambda_0)^2 \right)}{2(p +\imath\lambda_{0})}-\frac{\imath\arctan(-\lambda_0/(p+ \imath\lambda_0))}{p +\imath\lambda_{0}}+\frac{\imath\pi}{2(p+\imath\lambda_0)}\right]\,
 \end{multline}
 where we choose the negative real semi-axis as suitable branch cut for the definition of the logarithm, so that
 \[
  p=\rho e^{i\vartheta},\quad\vartheta\in(-\pi,\pi)\qquad\Rightarrow\qquad\log p:=\log\rho+i\vartheta.
 \]
 Furthermore, recalling that 
 \[
 \arctan z=\frac{1}{2\imath}\log\frac{1+\imath z}{1-\imath z}
 \]
 (see, e.g., \cite[Pag. 31]{G}), after some computations we can rewrite \eqref{laplace2} as
 \begin{equation}\label{eq:laplaceZ2}
 \widehat{Z}_{2}(p)=\frac{\imath\const_0}{4\pi}\,\frac{\log p -\log(-\lambda_0)-\imath\pi/2}{p +\imath\lambda_{0}}\,,
 \end{equation}
 Clearly, $\widehat{Z}_2$ is analytic and non vanishing on $\{p\in\C: \re(p)>0\}$. Then, by \eqref{eq:zeta_tr}, we conclude that $\widehat{q}$ can be extended by analyticity to the whole $\{p\in\C: \re(p)>0\}$.
\end{proof}

\begin{rem}
 \label{rem:zeta_2}
 In fact, $\widehat{Z}_{2}$ is analytic and non-vanishing also on the imaginary axis. Indeed the sole singularity $p=-\imath\lambda_0$ is easily seen to be removable.
 \end{rem}

Before discussing the extension of $\widehat{q}$ to the imaginary axis, we conclude this section rewriting \eqref{eq:charge_eq} in Laplace domain.

By the monochromatic assumption (with $\eta,c=0$),  \eqref{eq:charge_eq} reads
\begin{equation}
\label{charge-alpha}
q(t)+4\pi\sum_{k=-1,0,1}\beta_{k}\int_{0}^{t}\I(t-\tau)e^{-\imath\omega k\tau}q(\tau)\dtau=f(t),
\end{equation}
where 
\begin{equation}
\label{beta0}
\beta_{0}:= \frac{\gamma}{2\pi}- \frac{\log 2}{2\pi} -\frac{\imath}{8},\qquad \beta_{\pm1}:=\pm \frac{\imath\alpha_0}{2}\,,
\end{equation}
and $f$ is defined by \eqref{eq:for_comp}. Note that, as $\alpha(0)=0$,
\[
\lambda_{0}=4\,\imath\,e^{-4 \pi\beta_{0}-2\log2}.
\]
Now, for $\re(p)>1$, the Laplace transform of $\I$ is $\frac{1}{\log p}$ (see \cite[Sec. 18.3, Eq. (19)]{E}). Hence, we can apply the Laplace transform to \eqref{charge-alpha} so that
\begin{equation}
\label{laplace-charge-alpha}
\widehat{q}(p)+\frac{4\pi}{\log p}\sum_{k=-1,0,1}\beta_{k}\,\widehat q(p+\imath\omega k)=\widehat{f}(p),\qquad\re(p)>1.
\end{equation}
On the other hand, combining \eqref{eq:zeta_impl}, \eqref{eq:zeta_tr} and \eqref{eq:laplaceZ2} one sees that $\widehat{f}(p)=-\frac{4\pi\imath\widehat{Z}_2(p)}{\log p}$, that is
\begin{equation}
\label{in-laplace-charge-alpha}
\widehat{f}(p)=\frac{\const_{0}}{\log p}\left[\frac{\log p -\log(-\lambda_0)-\imath\pi/2}{p +\imath\lambda_{0}}\right],\qquad\re(p)>1.
\end{equation}
As a consequence, suitably rearranging terms, for every $\re(p)>1$,
\begin{equation}
\label{laplace-charge}
\widehat{q}(p)+\frac{2\pi\imath \alpha_0}{\log p+4\pi\beta_{0}}[\widehat{q}(p+\imath\omega\,)-\widehat{q}(p-\imath\omega\,) ]=\frac{\const_{0}}{\log p+4\pi\beta_{0}}\left[\frac{\log p -\log(-\lambda_0)-\imath\pi/2}{p +\imath\lambda_{0}}\right].
\end{equation}
As the sole singularities in \eqref{laplace-charge} are:
\begin{itemize}
 \item[--] the branch point of the $\log$ functions, i.e. $p=0$
 \item[--] the zero of $\log p+4\pi\beta_{0}$, i.e.
 \begin{equation}
  \label{eq-truepole}
  p_{\mathrm{s}}:=\imath e^{2(\log2-\gamma)},
 \end{equation}
\end{itemize}
\eqref{laplace-charge} is actually well-defined on the whole open right half-plane. In addition, by Proposition \ref{pro:q_halfplane}, $\widehat{q}$ is the unique analytic solution on the open right half-plane.


\subsection{Extension of the Laplace transform of the charge to the imaginary axis}
\label{subsec:inversion}

Here we deal with the extension of $\widehat{q}$, as solution of \eqref{laplace-charge}, to the imaginary axis.

The first step consists of decomposing $\widehat{q}$ according to horizontal strips of width $\omega$ in the complex plane. Precisely, for each $n\in\mathbb{Z}$, define
\begin{equation}
\label{eq-qn}
\widehat{q}_{n}(p):=\widehat{q}(p+\imath\omega n),\qquad p\in\S(\omega),
\end{equation}
with
\begin{equation}
\label{eq:strip}
\mathcal{S}(\omega):=\{p\in \mathbb{C}: 0\leq\im(p)<\omega\}.
\end{equation}
As a consequence, each $\widehat{q}_{n}(p)$ satisfies
\begin{multline}
\label{suspiria}
\widehat{q}_{n}(p)+\frac{2\pi\imath \alpha_0}{\log (p+\imath \omega n)+4\pi\,\beta_{0}}[\widehat{q}_{n+1}(p)-\widehat{q}_{n-1}(p) ] \\
 =\frac{\const_0}{\log (p+\imath \omega n)+4\pi\,\beta_{0}}\left[\frac{\log(p+\imath\omega n)-\log(-\lambda_0)-\imath\pi/2}{p +\imath\omega n+\imath\lambda_{0}}\right].
\end{multline}
Therefore, (\ref{laplace-charge}) can be rewritten in a compact form as 
\begin{equation}
\label{Poe}
\widehat{q}(p)=\mathcal{L}(p)\widehat{q}(p)+\widehat{g}(p),\qquad p\in\S(\omega),
\end{equation}
where we use (with a little abuse of notation) the identifications $\widehat{q}(p)=(\widehat{q}_n(p))_n$, $\widehat{g}(p)=(\widehat{g}_n(p))_n$ and set
\begin{equation}
\label{operator}
\begin{split}
&(\mathcal{L}(p)\widehat{q}(p))_{n}:=-\frac{2\pi\imath \alpha_0}{\log (p+\imath \omega n)+4\pi\,\beta_{0}}[\widehat{q}_{n+1}(p)-\widehat{q}_{n-1}(p) ]\\[.3cm]
& \widehat{g}_{n}(p):=\frac{\const_\alpha}{\log (p+\imath\omega n)+4\pi\,\beta_{0}}\left[\frac{\log(p+\imath\omega n)-\log(-\lambda_0)-\imath\pi/2}{p +\imath\omega n+\imath\lambda_{0}}\right]\,,\quad \mbox{for $n\neq0$\,.}
\end{split}
\end{equation}
Recalling Remark \ref{rem:zeta_2},\eqref{beta0} and \eqref{eq-truepole}, the coefficients of \eqref{operator} for $n\neq0$ fail to be analytic in $\mathcal S(\omega)$ at the points 
\begin{equation}
\label{singpoint}
\widetilde{p}=\imath \left( e^{2(\log2-\gamma)}-\omega\widetilde{n}\right),
\end{equation}
for some $\widetilde{n}=\widetilde{n}(\omega)\in\Z$, which are all on the imaginary axis. Since $\textrm{Im}(p)\in [0,\omega)$ whenever $p\in\mathcal{S}(\omega)$,
\begin{equation}
\label{bodega}
\frac{ e^{2(\log2-\gamma)}}{\omega}-1<\widetilde{n}\leqslant \frac{ e^{2(\log2-\gamma)}}{\omega},
\end{equation}
so that $\widetilde{n}$ must be, in fact, a natural number (recall that $\omega>0$).
\begin{rem}
 Note that, consistently, the point $\widetilde{p}$, defined by \eqref{singpoint}, and the point $p_{\mathrm{s}}$, defined by \eqref{eq-truepole}, are connected by the relation
 \[
  \widetilde{p}+\imath\omega\widetilde{n}=p_{\mathrm{s}}.
 \]
\end{rem}

 In addition one also sees that:
\begin{itemize}
 \item if  $N\,\omega=e^{2(\log2-\gamma)}$ for some $N\in\mathbb{N}$, then from \eqref{singpoint} and \eqref{bodega} $\widetilde{n}=N$, and thus $\widetilde{p}=0$;
 \item if, on the contrary,  $N\,\omega\neq e^{2(\log2-\gamma)}$ for every $N\in\mathbb{N}$, then one can find that $\widetilde{p}\neq0$ as follows:
 \begin{itemize}
  \item whenever $\omega>e^{2(\log2-\gamma)}$, condition \eqref{bodega} can be satisfied only for $\widetilde{n}=0$, and thus $\widetilde{p}=\imath\,e^{2(\log2-\gamma)}\neq0$,
  \item whenever $\omega<e^{2(\log2-\gamma)}$, condition \eqref{bodega} can be satisfied only for $\widetilde{n}$ equal to the integer part of $ e^{2(\log2-\gamma)}/\omega$, and thus $\widetilde{p}=\imath\ep\omega\neq0$, with $\ep$ equal to the fractional part of $ e^{2(\log2-\gamma)}/\omega$.
  \end{itemize}
\end{itemize}

From the previous section we know that every component of $\widehat{q}=(\widehat{q}_n)_{n\in\Z}$ is analytic on $\{p\in\S(\omega):\re(p)>0\}$. In the following, we prove that they can be extended by analiticity to $\S(\omega)\cap\imath\R$ up to the point $p=0$, which means that $\widehat{q}$, as a solution of \eqref{laplace-charge}, can be extended by analiticity to $\imath\R$, up to the points $p=\imath\omega n$.

As $p=0$ is the branch point of the logarithm, we split the analysis in two cases: $p\neq0$ and $p=0$.


\subsubsection{The case $p\neq0$.}

In view of the previous remarks, in the case $p\neq0$ it is convenient to carry on the discussion of \eqref{Poe} on $\S(\omega)\cap\imath\R$ by steps:
\begin{itemize}
\item[(i)] existence and analyticity for $p\neq0,\,\widetilde{p}$ (Proposition \ref{profondo rosso});
\item[(ii)] existence and analyticity for $p\neq0$ (Proposition \ref{tenebre}).
\end{itemize}

\begin{pro}[Step (i)]
\label{profondo rosso}
There exist an open and connected set $D(\omega)$, with $(\mathcal{S}(\omega)\cap\imath\R)\setminus\{0\}\subset D(\omega)\subset\mathcal{S}(\omega)$, and a discrete set $\mathcal{P}(\omega)\subseteq D(\omega)\setminus\imath\R$ such that for any $p\in D(\omega)\setminus\big(\mathcal{P}(\omega)\cup\{\widetilde{p}\}\big)$ equation \eqref{Poe} admits a unique solution $(\widehat{q}_n(p))_n\in \ell_{2}(\mathbb{Z})$ and, for each $n\in\mathbb{Z}$, $q_n$ is analytic on $D(\omega)\setminus\big(\mathcal{P}(\omega)\cup\{\widetilde{p}\}\big)$. In particular, all the functions $q_n$ are analytic on $(\mathcal{S}(\omega)\cap\imath\R)\backslash\{0, \widetilde{p}\}$.
\end{pro}

\begin{proof}
Preliminarily, observe that by \eqref{operator} $p\mapsto\mathcal{L}(p)$ is an analytic operator valued function in $D(\omega)\setminus\{\widetilde{p}\}$, with $D(\omega)$ a suitable open neighborhood of $(\mathcal{S}(\omega)\cap\imath\R)\setminus\{0\}$ contained in $\mathcal{S}(\omega)$. Note that $D(\omega)$ cannot contain $p=0$.

Then, as a first step, we show that, for any $p\in D(\omega)\setminus\{\widetilde{p}\}$, $\mathcal{L}(p)$ is a compact operator on $\ell_2(\Z)$. Clearly, it is a linear and bounded operator, being the composition of right and left shifts with the multiplication operator $B(p):\ell_2(\Z)\to\ell_2(\Z)$ associated with the sequence 
\[
B_n(p):=-\frac{2\pi\imath \alpha_0}{\log (p+\imath \omega n)+4\pi\,\beta_{0}}\,,\qquad n\in\Z\,,
\]
in such a way that
\[
(B(p)\widehat{q}\,)_n(p):=b_n(p)\widehat{q}_n(p)\,,\qquad n\in\Z\,.
\]
In addition, the compactness of $\mathcal{L}(p)$ immediately follows from the compactness of $B_p$. To this aim, we first see that $B(p)$ is the norm limit of a sequence of finite rank operators, due to the fact that $\lim_{n\to\infty}B_{n}(p)=0$. Indeed, this limit entails that
\[
\lim_{N\to\infty}\Vert B^N(p)-B(p)\Vert_{\ell_2(\Z)\to\ell_2(\Z)}=0\,,
\]
where $B^N(p)$ is the operator defined by
\[
(B^N(p)a)_n:=B_n(p)a_n,\quad\mbox{for $|n|\leq N$},\qquad (B^N(p)a)_n:=0,\quad \mbox{for $|n|>N$},
\]
for every $a=(a_n)_n\in\ell_2(\Z)$. Hence, from \cite[Theorems VI.12 and VI.13]{RSI}, $B(p)$ is compact.

In view of the previous remarks, in order to get the claim it suffices to apply the analytic Fredholm alternative to $\mathcal{L}(p)$ (see, e.g., \cite[Theorem VI.14]{RSI}). To this aim it suffices to show that the homogeneous equation associated with \eqref{Poe} has the sole trivial solution on $\ell_2(\Z)$, for every $p\in(\mathcal{S}(\omega)\cap\imath\R)\backslash\{0, \widetilde{p}\}$.

Therefore, fix $p\in(\mathcal{S}(\omega)\cap\imath\R)\backslash\{0, \widetilde{p}\}$, and set $\xi=\im(p)$. From \eqref{suspiria}, the homogeneous equation associated with \eqref{Poe} in components reads
\begin{equation}\label{eq:homdiscreteeq}
\widehat{q}_{n}(p)=-\frac{2\pi\imath \alpha_0}{\log (p+\imath \omega n)+4\pi\,\beta_{0}}[\widehat{q}_{n+1}(p)-\widehat{q}_{n-1}(p) ]\,,\qquad n\in\Z\,.
\end{equation}
Observe that $\log (p+\imath \omega n)=\log(\xi+\omega n)+\imath \frac{\pi}{2}$, and since $\beta_0:=-\frac{\log2}{2\pi}+\frac{\gamma}{2\pi}-\frac{\imath}{8}$, we get
\begin{equation}\label{eq:rewritelog}
\log (p+\imath \omega n)+4\pi\,\beta_{0}=\log(\xi+\omega n)-2\log2+2\gamma\,,
\end{equation}
so that $\log (p+\imath \omega n)\in\R$ for $n\in\mathbb{N}$.

Multiplying both sides of \eqref{eq:homdiscreteeq} by \eqref{eq:rewritelog} and taking the $\ell_2(\Z)$-product with $(\widehat{q}_n(p))_n$ we obtain
\begin{equation}\label{eq:series}
\sum_{n\in\Z}[\log(\xi+\omega n)-2\log2+2\gamma]\vert \widehat{q}_n(p)\vert^2=-2\pi\imath \alpha_0\sum_{n\in\Z}[\widehat{q}_{n+1}(p)\widehat{q}_{n}^*(p)-\widehat{q}_{n-1}(p)\widehat{q}_{n}^*(p)].
\end{equation}
Now, as
\[
 -2\pi\imath \alpha_0\sum_{n\in\Z}[\widehat{q}_{n+1}(p)\widehat{q}_{n}^*(p)-\widehat{q}_{n-1}(p)\widehat{q}_{n}^*(p)]=4\pi\alpha_0\sum_{n\in\Z}\im(\widehat{q}_{n+1}(p)\widehat{q}_{n}^*(p)),
\]
the r.h.s. of \eqref{eq:series} is real. Hence, l.h.s. must be real as well, but this may occur only if 
\begin{equation}
\label{eq-qneg}
\widehat{q}_{n}(p)=0\,,\qquad\forall n<0.
\end{equation}
However, as by \eqref{eq:homdiscreteeq}
\[
\widehat{q}_{n+1}(p)=-\frac{\log (p+\imath \omega n)+4\pi\,\beta_{0}}{2\pi\imath \alpha_0}\widehat{q}_{n}(p)+\widehat{q}_{n-1}(p)\,,\qquad \forall n\in\Z,
\]
\eqref{eq-qneg} actually imply that
\[
\widehat{q}_{n}(p)=0\,,\qquad\forall n\in\Z.
\]
We have thus proved that the homogeneous equation associated with \eqref{Poe} only admits the trivial solution on $(\mathcal{S}(\omega)\cap\imath\R)\backslash\{0, \widetilde{p}\}$. It is analytic there and we can apply the standard Fredholm alternative to ensures that $I-\mathcal{L}(p)$ is invertible for every $p\in(\mathcal{S}(\omega)\cap\imath\R)\backslash\{0, \widetilde{p}\}$, and then the analytic Fredholm to ensure that $(I-\mathcal{L}(p))^{-1}$ is analytic. Thus, as $\widehat{g}$ is analytic as well there, the proof is complete.
 \end{proof}

Before addressing subcase (ii), it is necessary to state following preliminary result.

\begin{lem}
\label{inferno}
Let $(\widehat{\xi}'_n(p))_n$ be a sequence such that
\[
\widehat{\xi}_{n}(p):=\frac{\widehat{\xi}_{n}'(p)}{\log (p+\imath \omega n)+4\pi\,\beta_{0}}
\]
belongs to $\ell_{2}(\mathbb{Z}\backslash\{\widetilde{n}\})$. If $\widetilde{p}\neq0$, then the system of equations 
\begin{equation}
\label{eqsing}
\begin{split}
\widehat{r}_{n}(p)&=-\frac{2\pi\imath \alpha_0}{\log (p+\imath \omega n)+4\pi\,\beta_{0}}\bigg\{\widehat{r}_{n+1}(p)-\widehat{r}_{n-1}(p)+\widehat{\xi}_{n}(p)\bigg\}\,,\quad\mbox{for $n\neq\widetilde{n},\widetilde{n}\pm1$}\,, \\
\widehat{r}_{\widetilde{n}+1}(p)&=-\frac{2\pi\imath \alpha_0}{\log (p+\imath \omega (\widetilde{n}+1))+4\pi\,\beta_{0}}\bigg\{\widehat{r}_{\widetilde{n}+2}(p)+\widehat{\xi}_{\widetilde{n}+1}(p)\bigg\}\,,\\
\widehat{r}_{\widetilde{n}-1}(p)&=-\frac{2\pi\imath \alpha_0}{\log (p+\imath \omega (\widetilde{n}-1))+4\pi\,\beta_{0}}\bigg\{-\widehat{r}_{\widetilde{n}-2}(p)+\widehat{\xi}_{\widetilde{n}-1}(p)\bigg\}\,,
\end{split}
\end{equation}
 has a unique solution $(\widehat{r}_{n}(p))_n\in \ell_{2}(\mathbb{Z}\backslash\{\widetilde{n}\})$ in a neighborhood of $\widetilde{p}$ up to a discrete set of points which are not contained in the imaginary axis. Moreover, if $\widehat{\xi}_{n}'$ is analytic in that set, then the same holds for the functions $\widehat{r}_{n}$.
\end{lem}

\begin{proof}
One can argue exactly as for Proposition \ref{profondo rosso}. Indeed, equation \eqref{eqsing} can be rewritten as
\[
\widehat{r}(p)=\mathcal{L}_1(p)\,\widehat{r}(p)+\widehat{\xi}(p)
\]
where $(\widehat{\xi}_n(p))_n\in \ell_{2}(\mathbb{Z}\backslash\{\widetilde{n}\})$ and $\mathcal{L}_1(p):\ell_{2}(\mathbb{Z}\backslash\{\widetilde{n}\})\to\ell_{2}(\mathbb{Z}\backslash\{\widetilde{n}\})$ is a compact operator. Then, in order to use again the analytic Fredholm theorem, one can show that, whenever $p$ belongs to a suitable imaginary neighborhood of $\widetilde{p}\neq0$, there is no non-trivial solution of the homogeneous equation associated with \eqref{eqsing}.

More precisely, multiplying the homogeneous equation associated with \eqref{eqsing} by $\log(\xi+\omega n)-2\log 2 + 2\gamma$ (where $\xi=\im(p)$), using \eqref{eq:rewritelog} and taking the $\ell_2(\Z\setminus\{\widetilde{n}\})$-product with $(\widehat{r}_n(p))_n$, there results
\begin{multline}
\label{eq-rn}
\sum_{n\in\Z\setminus\{\widetilde{n}\}}[\log(\xi+\omega n)-2\log2+2\gamma]\vert \widehat{r}_n(p)\vert^2\\
=-2\pi\imath \alpha_0\bigg(\sum_{n\in\Z\setminus\{\widetilde{n},\widetilde{n}\pm1\}}[\widehat{r}_{n+1}(p)\widehat{r}_{n}^*(p)-\widehat{r}_{n-1}(p)\widehat{r}_{n}^*(p)]+\widehat{r}_{\widetilde{n}+2}(p)\widehat{r}_{\widetilde{n}+1}^*(p)-\widehat{r}_{\widetilde{n}-2}(p)\widehat{r}_{\widetilde{n}-1}^*(p)\bigg)\\[.2cm]
=4\pi\alpha_0\bigg(\sum_{n\in\Z\setminus(-\infty,\widetilde{n}]}\im\big(\widehat{r}_{n+1}(p)\widehat{r}_{n}^*(p)\big)+\sum_{n\in\Z\setminus[\widetilde{n},+\infty)}\im\big(\widehat{r}_{n-1}(p)\widehat{r}_{n}^*(p)\big)\bigg).
\end{multline}
As a consequence $\widehat{r}_n(p)=0$, for every $n<0$. Moreover, exploiting again the homogeneous equation associated with \eqref{eqsing}, one can see that this extends to every $n\leq\widetilde{n}-1$.

It is then left to discuss the case $n\geq\widetilde{n}+1$. However, one can see that if $\widehat{r}_{\widetilde{n}+1}(p)\neq0$, then the sequence $(\widehat{r}_n(p))_n\not\in\ell_2(\Z\setminus\{\widetilde{n}\})$, arguing as follows. Assume 
\[
\widehat{r}_{\widetilde{n}+1}(p)\neq0\,,
\]
and consider the homogeneous version of \eqref{eqsing}. Then by the second equation in \eqref{eqsing} we find
\[
\widehat{r}_{\widetilde n+2}(p)=-\frac{\log (p+\imath \omega (\widetilde n+1))+4\pi\,\beta_{0}}{2\pi\imath\alpha_0}\widehat r_{\widetilde n+1}(p)\,,
\]
which, combined with the first equation in \eqref{eqsing}, gives
\[
\widehat r_{\widetilde n+3}(p)=\left[ \frac{(\log(p+\imath\omega(\widetilde n+1))+4\pi\beta_0)(\log(p+\imath\omega(\widetilde n+2))+4\pi\beta_0)}{(2\pi\imath\alpha_0)^2}\right]\widehat r_{\widetilde n+1}(p)\,.
\]
Iterating this procedure one finds an expression of the form
\[
\widehat r_{\widetilde n+ k}= \Lambda(p,\alpha, k)\, \widehat r_{\widetilde n+1}(p)\,,\qquad \forall k\in\N\,,
\]
where
\[
\lim_k \vert \Lambda(p,\alpha_0,k)\vert=+\infty\,,
\]
thus contradicting the fact that $(\widehat{r}_n(p))_n\in\ell_2(\Z\setminus\{\widetilde{n}\})$. Thus
\[
 \widehat r_{\widetilde n+1}(p)=0
\]
and the recursive structure of the homogeneous equation associated with \eqref{eqsing} implies that $\widehat{r}_n(p)=0$, for every $n\geq\widetilde{n}+1$, thus proving the claim.
\end{proof}

\begin{pro}[Step (ii)]
\label{tenebre}
There exists a unique solution $(\widehat{q}_{n}(p))_n\in \ell_{2}(\mathbb{Z})$ of \eqref{Poe} for every $p\in D(\omega)\setminus\mathcal{P}(\omega)$, where $D(\omega)$ is a suitable open and connected set such that $(\mathcal{S}(\omega)\cap\imath\R)\setminus\{0\}\subset D(\omega)\subset\mathcal{S}(\omega)$ and $\mathcal{P}(\omega)$ is a discrete set of points in $D(\omega)$ not intersecting the imaginary axis. In addition, the functions $\widehat{q}_n$ are analytic on $D(\omega)\setminus\mathcal{P}(\omega)$.
\end{pro}

\begin{proof}
Proposition \ref{profondo rosso} actually proves the claim except for the point $\widetilde{p}$. Hence, in the following we only focus on what happens in a neighborhood of $\widetilde{p}$ whenever $\widetilde{p}\neq0$.

If $e^{2(\log2-\gamma)}=N\,\omega$ for some $N\in\mathbb{N}$ is satisfied, then the claim is trivial since $\widetilde{p}=0$, which is out of $D(\omega)$ (how highlighted in the proof of Proposition \ref{profondo rosso}). On the contrary, if $e^{2(\log2-\gamma)}\neq N\,\omega$ for every $N\in\mathbb{N}$, then $\widetilde{p}\neq0$ and the proof requires some further effort. In particular, we have to prove that each $\widehat{q}_n$ can be extended to $\widetilde{p}$ and is analytic here.

The strategy of the proof consists of discussing separately the terms $\widehat{q}_n$, with $n\neq\widetilde{n}$, and $\widehat{q}_{\widetilde{n}}$. Let $\widehat{t}_n(p)$ be the \emph{non-trivial} solution in $\ell_2(\Z\setminus\{\widetilde{n}\})$ provided by Lemma \ref{inferno} of the system
\begin{equation}
\label{eqs}
\begin{split}
\widehat{t}_{n}(p)&=-\frac{2\pi\imath \alpha_0}{\log (p+\imath \omega n)+4\pi\,\beta_{0}}\bigg\{\widehat{t}_{n+1}(p)-\widehat{t}_{n-1}(p)\bigg\}\,,\quad\mbox{for $n\neq\widetilde{n},\widetilde{n}\pm1$}\,, \\
\widehat{t}_{\widetilde{n}+1}(p)&=-\frac{2\pi\imath \alpha_0}{\log (p+\imath \omega (\widetilde{n}+1))+4\pi\,\beta_{0}}\bigg\{\widehat{t}_{\widetilde{n}+2}(p)-1\bigg\}\,,\\
\widehat{t}_{\widetilde{n}-1}(p)&=-\frac{2\pi\imath \alpha_0}{\log (p+\imath \omega (\widetilde{n}-1))+4\pi\,\beta_{0}}\bigg\{-\widehat{t}_{\widetilde{n}-2}(p)+1\bigg\}\,,
\end{split}
\end{equation}
Set also
\begin{equation}
\label{eq:qdecomp}
\widehat{\rho}_n(p):=\widehat{q}_n(p)-\widehat{t}_n(p)\widehat{q}_{\widetilde{n}}(p),\qquad\forall n\in\Z\,, n\neq\widetilde{n}.
\end{equation}
Plugging into \eqref{suspiria} (and in view of \eqref{operator}), there results that, for every $n\neq\widetilde{n}$,
\begin{multline*}
\widehat{\rho}_{n}(p)+\widehat{t}_{n}(p)\,\widehat{q}_{\widetilde{n}}(p)=\widehat{g}_n(p)+\\
-\frac{2\pi\imath \alpha_0}{\log (p+\imath \omega n)+4\pi\,\beta_{0}}\bigg\{\widehat{\rho}_{n+1}(p)+\widehat{t}_{n+1}(p)\,\widehat{q}_{\widetilde{n}}(p)-\widehat{\rho}_{n-1}(p)-\widehat{t}_{n-1}(p)\,\widehat{q}_{\widetilde{n}}(p)\bigg\},
\end{multline*}
As a consequence, combining with \eqref{eqs} and \eqref{eq:qdecomp} (for $n=\widetilde{n}$), $(\widehat{\rho}_n(p))_n$ must satisfy
\begin{equation}
\label{eqsb}
\widehat{\rho}_{n}(p)=-\frac{2\pi\imath \alpha_0}{\log (p+\imath \omega n)+4\pi\,\beta_{0}}\bigg\{ \widehat{\rho}_{n+1}(p)-\widehat{\rho}_{n-1}(p)\bigg\}+\widehat{g}_n(p)\,,\quad\mbox{for $n\neq,\widetilde{n},\,\widetilde{n}\pm1$,}
\end{equation}
and
\begin{equation}
\label{eqsb2}
\widehat{\rho}_{\widetilde{n}\pm1}(p)=\mp\frac{2\pi\imath \alpha_0}{\log (p+\imath \omega(\widetilde{n}\pm1))+4\pi\,\beta_{0}}\widehat{\rho}_{\widetilde{n}\pm2}(p)+\widehat{g}_{\widetilde{n}\pm1}(p).
\end{equation}
Now, one can easily check that the argument used in the proof of Lemma \ref{inferno} applies to \eqref{eqsb}-\eqref{eqsb2} as well (since $\overline{p}\in\imath\R$). Thus, there exists a unique solution $(\widehat{\rho}_{n}(p))_n\in\ell_2(\Z\setminus\{\widetilde{n}\})$ in a neighborhood of $\widetilde{p}$ up to a discrete set of points which are not contained in the imaginary axis and all the $\widehat{\rho}_{n}$ are analytic here.

Combining the previous remarks with Lemma \ref{inferno} and \eqref{eq:qdecomp}, one clearly sees that it is left to prove that $\widehat{q}_{\widetilde{n}}$ exists and is analytic in a neighborhood of $\widetilde{p}$. To this aim, we first note that $\widehat{q}_{\widetilde{n}}(p)$ has to solve the following equation
\[
\widehat{q}_{\widetilde{n}}(p)+\frac{2\pi\imath \alpha_0}{\log (p+\imath \omega \widetilde{n})+4\pi\,\beta_{0}} \bigg\{ \widehat{\rho}_{\widetilde{n}+1}(p)+\widehat{t}_{\widetilde{n}+1}(p)\widehat{q}_{\widetilde{n}}(p)- \widehat{\rho}_{\widetilde{n}-1}(p)-\widehat{t}_{\widetilde{n}-1}(p)\widehat{q}_{\widetilde{n}}(p) \bigg\}=\widehat{g}_{\widetilde{n}}(p),
\]
namely,
\begin{multline*}
\big[\log (p+\imath \omega \widetilde{n})+4\pi\beta_{0}+2\pi\imath \alpha_0(\widehat{t}_{\widetilde{n}+1}(p)-\widehat{t}_{\widetilde{n}-1}(p)\big]\widehat{q}_{\widetilde{n}}(p)\\
=-2\pi\imath \alpha_0(\widehat{\rho}_{\widetilde{n}+1}(p)-\widehat{\rho}_{\widetilde{n}-1}(p))+\left(\log (p+\imath \omega \widetilde{n})+4\pi\beta_{0}\right)\widehat{g}_n(p).
\end{multline*}
Note that the r.h.s. of the above equation is analytic in a neighborhood of $\widetilde{p}$, as well as $\widehat{\rho}_{\widetilde{n}+1},\,\widehat{\rho}_{\widetilde{n}-1},\,\widehat{t}_{\widetilde{n}+1},\,\widehat{t}_{\widetilde{n}-1}$. Moreover, by definition, 
\[
\log (\widetilde{p}+\imath \omega \widetilde{n})+4\pi\beta_{0}=0\,,
\]
 so that in order to get existence and analyticity of $\widehat{q}_{\widetilde{n}}$ it is sufficient to prove that 
\begin{equation}
 \label{eq-ipabs}
\widehat{t}_{\widetilde{n}+1}(\widetilde{p})\neq\widehat{t}_{\widetilde{n}-1}(\widetilde{p})\,.
\end{equation}
Therefore, assume by contradiction that $\widehat{t}_{\widetilde{n}+1}(\widetilde{p})=\widehat{t}_{\widetilde{n}-1}(\widetilde{p})$. Thus, multiplying \eqref{eqs} evaluated at $\widetilde{p}$ by $(\log (\widetilde{p}+\imath \omega n)+4\pi\,\beta_{0})$, taking the $\ell_2(\Z\setminus\{\widetilde{n}\})$-product with $(\widehat{t}_n(\widetilde{p}))_n$ and arguing as in \eqref{eq-rn}, there results
\begin{multline}
\label{eqsi}
\sum_{n\in\Z\setminus\{\widetilde{n}\}}(\log (\widetilde{p}+\imath \omega n)+4\pi\,\beta_{0})|\widehat{t}_{n}(\widetilde{p})|^{2}=2\pi\imath \alpha_0\underbrace{\big(\widehat{t}_{\widetilde{n}+1}(\widetilde{p})-\widehat{t}_{\widetilde{n}-1}(\widetilde{p})\big)}_{=0}+\\[.2cm]
4\pi\alpha_0\bigg(\sum_{n\in\Z\setminus(-\infty,\widetilde{n}-1]}\im\big(\widehat{t}_{n+1}(p)\widehat{t}_{n}^*(p)\big)+\sum_{n\in\Z\setminus[\widetilde{n}+1,+\infty)}\im\big(\widehat{t}_{n-1}(p)\widehat{t}_{n}^*(p)\big)\bigg)
\end{multline}
Clearly, as the last term at the right hand side is real, since by \eqref{beta0}
\[
\log (\widetilde{p}+\imath \omega n)+4\pi\,\beta_{0}=\log (\omega n-\imath\widetilde{p} )+2(\gamma-\log2)
\]
and since $\widetilde{p}$ is purely imaginary, one obtains that $\widehat{t}_n(\widetilde{p})=0$ for every $n<\frac{\imath\widetilde{p}}{\omega}$. Furthermore, as $\im(\widetilde{p})\in(0,\omega)$, $\frac{\imath\widetilde{p}}{\omega}\in(-1,0)$, so that $\widehat{t}_n(\widetilde{p})=0$ for all $n<0$. However, by the first line of \eqref{eqs} (used for every $n\leq\widetilde{n}-2$), this actually entails that $\widehat{t}_n(\widetilde{p})=0$ for all $n\leq\widetilde{n}-1$, but this contradicts the third line of \eqref{eqs} since $\alpha_0\neq0$. As a consequence \eqref{eq-ipabs} is true, which completes the proof.
\end{proof}

\begin{rem}\label{rmk:ontheaxis}
 Combining Propositions \ref{pro:q_halfplane}, \ref{profondo rosso} and \ref{tenebre} one obtains that the functions $\widehat{q}_n$ are analytic up to the imaginary axis, except possibly at $p=0$, so that the possible singularities given by the set $\mathcal{P}(\omega)$ must belong to the left open half-plane $\{p\in\C\,:\,\re(p)<0\}$. In addition, the analytic Fredholm alternative \cite[Theorem VI.14]{RSI} guarantee that they must be poles.
 \end{rem}
 

\subsubsection{The case $p=0$.}
\label{subsubsec:pzero}

All the results obtained up to this point do not require any further assumption on the frequency $\omega$ of $\alpha(t)$. On the contrary, at this point we have to use the so-called \emph{no-resonance} condition \eqref{eq:nonres}, that is
\[
N\omega\neq e^{2(\log2-\gamma)},\qquad\forall N\in\mathbb{N}.
\]
Recall that, as explained before \eqref{bodega}, this entails that $\widetilde{p}\neq0$.
\begin{rem}
 It is worth recalling why we call \eqref{eq:nonres} no-resonance assumption. Indeed, as the monochromatic assumption entails $\alpha(0)=0$, $\lambda_0=-4e^{-2\gamma}$ and thus \eqref{eq:nonres} reads $N\omega\neq-\lambda_0$, for every $N\in\N$. In other words, \eqref{eq:nonres} implies that the frequency $\omega$ is not an integer divisor of the eigenvalue of $\hamz$.
\end{rem}

\begin{pro}
\label{nonhosonno}
If condition \eqref{eq:nonres} holds, then the solution of \eqref{Poe} is of the form
\begin{equation}
\label{ilgattoanovecode} 
 \left\{\begin{array}{l}
  \displaystyle \widehat{q}_{0}(p)=\frac{H(p)-4\pi\imath\widehat{Z}_2(p)}{Q(p)+\log p}, \\[.7cm]
  \displaystyle \widehat{q}_n(p)=\widehat{R}_n(p)+\widehat{\tau}_n(p)\,\widehat{q}_0(p),\qquad n\neq0,
 \end{array}\right.
\end{equation}
where all the functions $\widehat{R}_n$, $\widehat{\tau}_n$, $H$ and $Q$ are analytic at $p=0$.
\end{pro}

\begin{proof}
 Using the fact that \eqref{eq:nonres} implies that $\widetilde{p}\neq0$ and arguing as in the proofs of Propositions \ref{profondo rosso} and \ref{tenebre} and Lemma \ref{inferno}, one can check that, up to a discrete set of points, in a neighborhood of $p=0$ there exist a unique solution $(\widehat{\tau}_{n}(p))_n\in \ell^{2}(\mathbb{Z}\backslash\{0\})$ of
 \begin{equation}
 \label{eq-taumenouno}
 \begin{split}
\widehat{\tau}_{n}(p)&=-\frac{2\pi\imath \alpha_0}{\log (p+\imath \omega n)+4\pi\,\beta_{0}}\bigg\{\widehat{\tau}_{n+1}(p)-\widehat{\tau}_{n-1}(p)\bigg\}\,,\quad\mbox{for $n\neq 0,\pm1$}\,, \\
\widehat{\tau}_{\pm1}(p)&=-\frac{2\pi\imath \alpha_0}{\log (p\pm\imath \omega)+4\pi\,\beta_{0}}\bigg\{\pm\widehat{\tau}_{\pm2}(p)\mp1\bigg\}\,,
\end{split}
\end{equation}
and a unique solution $(\widehat{R}_{n}(p))_n\in \ell^{2}(\mathbb{Z}\backslash\{0\})$ of
\begin{equation}
\label{eq-erremenouno}
 \begin{split}
\widehat{R}_{n}(p)&=-\frac{2\pi\imath \alpha_0}{\log (p+\imath \omega n)+4\pi\,\beta_{0}}\bigg\{\widehat{R}_{n+1}(p)-\widehat{R}_{n-1}(p)\bigg\}+\widehat{g}_n(p)\,,\quad\mbox{for $n\neq 0,\pm1$}\,, \\
\widehat{R}_{\pm1}(p)&=\mp\frac{2\pi\imath \alpha_0}{\log (p\pm\imath \omega)+4\pi\,\beta_{0}}\widehat{R}_{\pm2}(p)+\widehat{g}_{\pm1}(p).
\end{split}
\end{equation}
As a consequence, arguing again as in the proof of Proposition \ref{tenebre}, one finds that, if $\widehat{q}_0(p)$ solves \eqref{suspiria} for $n=0$, then all the solutions for $n\neq0$ are given by $\widehat{q}_n(p)=\widehat{R}_n(p)+\widehat{\tau}_n(p)\,\widehat{q}_0(p)$.

Therefore, it is left to discuss the existence and the behavior of $\widehat{q}_0(p)$. Setting $n=0$ in \eqref{suspiria} and using the previous decomposition with $n=\pm1$, there results that
\begin{align*}
\left(\log (p)+4\pi\beta_{0}+2\pi\imath \alpha_0[\widehat{\tau}_{1}(p)-\widehat{\tau}_{-1}(p)] \right)\widehat{q}_{0}(p)=-2\pi\imath \alpha_0[\widehat{R}_{1}(p)-\widehat{R}_{-1}(p)]+\widehat{G}_0(p),
\end{align*}
where, by \eqref{operator} and \eqref{eq:laplaceZ2},
\[
 \widehat{G}_0(p)=-4\pi\imath \widehat{Z}_2(p).
\]
As a consequence, setting
\begin{equation}
\label{eq-acca}
H(p):=-2\pi\imath \alpha_0[\widehat{R}_{1}(p)-\widehat{R}_{-1}(p)],
\end{equation}
and
\begin{equation}
\label{eq-cu}
Q(p):=4\pi\beta_{0}+2\pi\imath \alpha_0[\widehat{\tau}_{1}(p)-\widehat{\tau}_{-1}(p)].
\end{equation}
one gets the claim.
\end{proof}

\begin{rem}
 \label{rem-whynonres}
 The proof of Proposition \ref{nonhosonno} strongly relies on the no-resonance assumption \eqref{eq:nonres}. Indeed, the strategy used before works only if the singularity of the coefficients of the equations is not located at $p=0$. In the one-dimensional case (see \cite{CCLR}) a regularization technique has been proposed in order to overcome such an issue, but it is not clear to us how to adapt it to this case due to the logarithmic nature of the singularities. More precisely, in odd dimensions coefficients can be regularized simply by squaring, as the singularity is due to the presence of a square root.  It is conceivable that in two dimensions a local exponential map would straighten the log-singularity, thus allowing to perform a detailed analysis.
  However, this interesting and challenging technical issue deserves a further specific investigation, and we plan to address it in a forthcoming paper. 
\end{rem}
\begin{rem}
\label{rem-sameform}
Notice that \eqref{ilgattoanovecode} with analytic $H,Q,\hat{R}_n,\hat{\tau}_n$ holds in a neighborhood of $p=0$.  Adapting the above arguments relying on the analytic Fredholm alternative one could prove that the same representation holds in the whole left half-plane, but a priori analyticity of the mentioned coefficients may fail at a discrete set of points. For this reason, in what follows we study $\hat{q}_n$ in the left half-plane using a different strategy.
\end{rem}

\subsection{Extension of the Laplace transform of the charge to the open left half-plane}
\label{subsec:lefthalfplane}

In the previous sections we proved that $\widehat{q}$ can be extended analitically to the whole closed right half-plane, except for the points $\imath\omega n$, for all $n\in\Z$, where logarithmic branch points are located. Here, we complete the analysis discussing the further extension to the open left half-plane. More precisely, we show that it is analytic on the open left half-plane except, possibly, for infinitely many simple poles $(p_n(\alpha_0,\omega))_{n\in\Z}$ of the form
\begin{equation}
\label{eq-expansion}
 p_n(\alpha_0,\omega)=p_0(\alpha_0,\omega)+\imath\omega n,\qquad\text{with}\quad\re(p_0(\alpha_0,\omega))<0,\quad\im(p_0(\alpha_0,\omega))\in(0,\omega),
\end{equation}
and infinitely many branch cuts starting at $\imath\omega n$, for all $n\in\Z$. For the sake of simplicity, we denote the set of all the branch cuts as
\begin{equation}
\label{eq-branchcuts}
 B(\omega):=\{p\in\C:\re(p)<0\:\text{ and }\:\im(p)=\omega n\:\text{ for some }\:n\in\Z\}.
\end{equation}

Let us start with the following preliminary remarks. As mentioned before, in Sections \ref{subsec:Laplace} and \ref{subsec:inversion}, we proved that there exists a unique function $\widehat{q}$ that solves \eqref{laplace-charge} in the closed right half-plane and is analytic on $\{p\in\C:\re(p)\geq0\}\setminus\{\imath\omega n:n\in\Z\}$. It is, then, left to discuss the possibility of extending it to the open letf half-plane.

However, as in the previous section, it is convenient to use the operator form of \eqref{laplace-charge} given by \eqref{Poe}. From this point of view, thus far we have proved that there exists a unique sequence of functions $\widehat{q}=(\widehat{q}_n)_n$ (defined as in \eqref{eq-qn}) such that:
\begin{itemize}
 \item[(i)] each $q_n$ is analytic on $\S_r(\omega)\setminus\{0\}$, with $\S_r(\omega):=\{p\in\S(\omega):\re(p)\geq0\}$;
 
 \item[(ii)] for every $p\in\S_r(\omega)$, $q(p)=(\widehat{q}_n(p))_n\in\ell_2$ and is given by
 \begin{equation}
  \label{eq-inversion}
  \widehat{q}(p)=(\ID-\mathcal{L}(p))^{-1}\,\widehat{g}(p),
 \end{equation}
 (recall that the version of $\widehat{g}=(\widehat{g}_n)_n$ as a sequence of functions is given by \eqref{operator}).
\end{itemize}
Therefore, the goal of this section is to investigate the possibility of extending the right hand side of \eqref{eq-inversion} to $\S_\ell(\omega)$, with
\[
 \S_\ell(\omega):=\{p\in\C:\re(p)<0,\:\im(p)\in[0,\omega]\}.
\]

\begin{rem}
 Note that in this case we have to take into account also points with imaginary part equal to $\omega$ as it is necessary to guarantee that the jumps at the branch cuts are finite. In the right half-plane this is not necessary, since $\widehat{q}$ is analytic everywhere up to the branch points.
\end{rem}

To this aim, one can see first that $\widehat{g}=(\widehat{g}_n)_n$ is well defined and non-vanishing in $\S_\ell(\omega)$ and analytic on $\IN{\S_\ell(\omega)}$, so that in order to extend $\widehat{q}=(\widehat{q}_n)_n$ it suffices to study behavior of $(\ID-\mathcal{L}(p))^{-1}$ here. Indeed, if it is analytic, then (using the analytic Fredholm alternative) $\widehat{q}=(\widehat{q}_n)_n$ is analytic. On the other hand, if $(\ID-\mathcal{L}(p))^{-1}$ does possess a pole $p_0$, then $\widehat{q}=(\widehat{q}_n)_n$ may display a pole here and thus, by \eqref{eq-qn}, $\widehat{q}$ may possess infinitely many poles of the form \eqref{eq-expansion}, as a single functions. Note that, clearly $(\ID-\mathcal{L}(p))^{-1}$ cannot be analytic in $\S_\ell(\omega)\setminus\IN{\S_\ell(\omega)}$ due to logarithmic branch cuts. However, if it has no pole here, then $\widehat{q}$, although not being analytic, is well defined. Hence, as a single function, has finite jumps at the branch cuts.

As a consequence, in this section we focus on possible
\begin{center}
 poles of $(\ID-\mathcal{L}(p))^{-1}$ in $\S_\ell(\omega)$,
\end{center}
namely, values of
\begin{center}
 $p\in\S_\ell(\omega)$ such that $\ker(\ID-\mathcal{L}(p))\neq\{0\}$.
\end{center}

\begin{rem}
Throughout the section, for the sake of simplicity, we will use the notation
\begin{equation}\label{eq:h}
h(p):=\frac{2\pi\imath}{\log p+4\pi\beta_0},
\end{equation}
so that $(\mathcal{L}(p)\widehat{q}(p))_n=-\alpha_0h(p+\imath\omega n)\big[\widehat{q}_{n+1}(p)-\widehat{q}_{n-1}(p)\big]$.
\end{rem}


\subsubsection{Poles of $(\ID-\mathcal{L}(p))^{-1}$ for small $\alpha_0$}

Following \cite{CCL}, our strategy consists of focusing first on small values of $\alpha_0$ and then of extending the results to arbitrary values of $\alpha_0$. This subsection addresses the former case. In the sequel, we will always explicit also the dependence of $\cL$ on $\alpha_0$.

\begin{teo}
\label{teo:poles_smallalpha}
Assume that \eqref{eq:nonres} is satisfied. Then, there exists $A_0>0$ such that, whenever $0<\vert \alpha_0\vert<A_0$, there is at most one pole $p_0=p_0(\alpha_0,\omega)$ of $(\ID-\mathcal{L}(\alpha_0,p))^{-1}$ on $\S_\ell(\omega)$. In particular, if such $p_0$ does exists, then it actually belongs to $\IN{\S_\ell(\omega)}$ and depends analytically on $\alpha_0$. As a consequence, $\widehat{q}$ can be extended to the open left half-plane, up to has at most a sequence of poles defined as in \eqref{eq-expansion}.
\end{teo}

Previously to the proof of Theorem \ref{teo:poles_smallalpha} some auxiliary results are required. As extensively explained before, in order to prove the theorem, it is sufficient to prove that there may exists at most one point $p_0\in\S_\ell(\omega)$ such that the homogeneous version of \eqref{Poe}, namely
\begin{equation}
\label{eq:homogeneous_eq}
\widehat{\xi}(p)=\mathcal{L}(\alpha_0,p)\widehat{\xi}(p),
\end{equation}
admits a nontrivial solution in $\ell_2$ for some $p=p_0$, and that it actually belongs to $\IN{\S_\ell(\omega)}$.

For such purposes, it is useful to state the following result

\begin{lem}
\label{lem-cisalva}
 If there exists $p_0\in\S_\ell(\omega)$ such that there is $0\neq(\xi_n(p_0))_n\in\ell_2$ that solves \eqref{eq:homogeneous_eq}, then $\xi_n(p_0)\neq0$, for every $n\in\Z$.
\end{lem}

\begin{proof}
Notice that the proof of this Lemma retraces that of Lemma \ref{inferno}. If the assumptions of the Lemma are satisfied, then there exists $n^*\in\Z$, such that $\xi_{n^*}(p_0)\neq0$. Assume, by contradiction, that $\xi_{n^*-1}(p_0)=0$. Then, by \eqref{eq:homogeneous_eq}, and recalling \eqref{operator}, the equation for the $n^*$-th component gives
\[
\xi_{n^*+1}(p_0)=-\alpha_0^{-1}h^{-1}(p_0+\imath\omega n^*)\xi_{n^*}(p_0)\,.
\] 
Moreover, looking at the equation for $\xi_{n*+1}$ we find
\[
\xi_{n^*+2}(p_0)=[\alpha^{-2}_0h(p_0+\imath\omega n^*)h(p_0+\imath\omega (n^*+1))]\xi_{n^*}(p_0) \,.
\]
Iterating this procedure we find a relation of the form
\[
\xi_{n^*+k}(p_0)=F(p_0,\alpha_0,k)\xi_{n^*}\,,\qquad \forall k\in\N\,,
\]
where $\lim_k\vert F(p+0,\alpha_0,k)\vert=+\infty$. Then $(\xi_n(p_0))_n\not\in\ell_2$, reaching a contradiction. Hence, $\xi_{n^*-1}(p_0)\neq0$. However, we can now repeat the argument for $\xi_{n^*-2}(p_0)$ and so on, thus obtaining that $\xi_{n}(p_0)\neq0$, for every $n\leq n^*$.

Assume then, again by contradiction, that $\xi_{n^*+1}(p_0)=0$. Then, arguing as before one sees that $(\xi_n(p_0))_n\not\in\ell_2$, which contradicts the assumptions. Hence, as before, one obtains that $\xi_{n}(p_0)\neq0$, for every $n> n^*$, thus concluding the proof.
\end{proof}

An immediate consequence of Lemma \ref{lem-cisalva} is that, whenever \eqref{eq:homogeneous_eq} has a non-trivial solution at $p\in\S_\ell(p)$, then $\xi$, seen as a single function, does not vanish at $p$ and satisfies
\begin{equation}
\label{eq:homogeneous-exp}
\widehat{\xi}(p)=-\alpha_0h(p)[\widehat{\xi}(p+\imath\omega)-\widehat{\xi}(p-\imath\omega)].
\end{equation}
Therefore, to detect possible poles of $(\ID-\cL(\alpha_0,p))^{-1}$ it is sufficient to detect points in which \eqref{eq:homogeneous-exp} has a non trivial solution.

\medskip
To this aim, we first note that if \eqref{eq:homogeneous_eq} has a nontrivial solution at $p\in\S_\ell(p)$, then (in view of Lemma \ref{lem-cisalva} and \eqref{eq:homogeneous-exp}) one can define the functions
\begin{equation}\label{eq:rho}
\rho(p):=\frac{\widehat{\xi}(p)}{\widehat{\xi}(p-\imath\omega)}
\end{equation}
and
\begin{equation}\label{eq:Omega}
\Omega(p):=\frac{\widehat{\xi}(p-\imath\omega)}{\widehat{\xi}(p)},
\end{equation}
 which have to solve the following equations
\begin{equation}\label{eq:eq_rho}
\rho= \mathcal{N}(\rho)\,,\qquad\mathcal{N}(\rho):=\frac{\alpha_0h}{1+\alpha_0h\rho(\cdot+\imath\omega)},
\end{equation}
\begin{equation}\label{eq:eq_Omega}
\Omega=\mathcal{M}(\Omega),\qquad\mathcal{M}(\Omega):=-\frac{\alpha_0h(\cdot-\imath\omega)}{1-\alpha_0h(\cdot-\imath\omega)\Omega(\cdot-\imath\omega)}.
\end{equation}

Let us, first, study equations \eqref{eq:eq_rho} and \eqref{eq:eq_Omega} separately. Given $\delta>0$, define
\[
J_{\delta}:=\{p\in\C:\text{Im}(p)\geq \delta,\: \re(p)\in[0,1] \}
\]
and
\[
K_{\delta}:=\{p\in\C:\text{Im}(p)\leq -\delta,\: \re(p)\in[0,1] \}
\]
and denote by $C(J_\delta),\,C(K_\delta)$ the spaces of continuous functions on $J_\delta,\,K_\delta$, endowed with the sup-norm. Then, one can establish the next

\begin{lem}
\label{lem:contraction}
Given $M>0$, there exist $A_0>0$ and suitable $\delta_1,\delta_2>0$ such that, if $0<\vert \alpha_0\vert<A_0$, then $\mathcal{N}$ in \eqref{eq:eq_rho} is a contraction on $\{\Vert \rho\Vert_{ C(J_{\delta_1})}\leq M\}\subset C(J_{\delta_1})$, and $\mathcal{M}$ in \eqref{eq:eq_Omega} is a contraction on $\{\Vert\Omega\Vert_{ C(K_{\delta_2})}\leq M\}\subset C(K_{\delta_2})$. Hence, \eqref{eq:eq_rho} has a unique solution on $\{\Vert \rho\Vert_{ C(J_{\delta_1})}\leq M\}$ and \eqref{eq:eq_Omega} has a unique solution on $\{\Vert \Omega\Vert_{ C(K_{\delta_2})}\leq M\}$. Moreover, both these solutions are analytic.
\end{lem}

\begin{proof}
We just focus on $\mathcal{N}$ and equation \eqref{eq:eq_rho}, as the proof for $\mathcal{M}$ and \eqref{eq:eq_Omega} follows along the same lines. By \eqref{eq:h}, one immediately sees that for $\delta_1$ large, there holds $\vert h\vert<1/2$ on $J_\delta$. Consider, then, $\rho\in C(J_{\delta_1})$, with $\Vert\rho\Vert_{ C(J_{\delta_1})}\leq M$. As a consequence,
\[
\Vert\mathcal{N}(\rho)\Vert_{ C(J_{\delta_1})}\leq\frac{1}{2}\frac{\vert\alpha_0\vert}{(1-\vert\alpha_0\vert M/2)}\leq M\,,
\]
for $\vert\alpha_0\vert\leq \frac{2M}{1+M^2}$, so that $\mathcal{N}$ preserves the ball $\{\Vert \rho\Vert_{ C(J_{\delta_1})}\leq M\}$. Moreover, easy computations yields
\[
 \mathcal{N}(\rho_1)-\mathcal{N}(\rho_2)=\alpha_0^2h^2\,\frac{\rho_2(\cdot+\imath\omega)-\rho_2(\cdot+\imath\omega)}{\big(1+\alpha_0h\rho_1(\cdot+\imath\omega)\big)\big(1+\alpha_0h\rho_2(\cdot+\imath\omega)\big)},\qquad\forall\rho_1,\,\rho_2\in C(J_{\delta_1}),
\]
so that
\[
 \|\mathcal{N}(\rho_1)-\mathcal{N}(\rho_2)\|_{ C(J_{\delta_1})}\leq\frac{\alpha^2_0}{4\big(1-\vert\alpha_0\vert M/2\big)^2}\,\|\rho_1-\rho_2\|_{ C(J_{\delta_1})},\qquad\forall\rho_1,\,\rho_2\in C(J_{\delta_1}),
\]
which shows that $\mathcal{N}$ is a contraction on $\{\Vert \rho\Vert_{ C(J_{\delta_1})}\leq M\}$, for $\vert\alpha_0\vert$ smaller than $\frac{2}{1+M}$. Whence, \eqref{eq:eq_rho} has a unique solution on $\{\Vert \rho\Vert_{ C(J_{\delta_1})}\leq M\}$ by the Banach-Caccioppoli fixed oint theorem. In addition, since analiticity is preserved by uniform convergence and Morera's theorem, one can repeat the contraction argument in the subset of the analytic function of $\{\Vert \rho\Vert_{ C(J_{\delta_1})}\leq M\}$, thus proving the solution to be analytic.
\end{proof}
\begin{rem}
Notice that in the above Lemma one can take 
\[
A_0=\min\{2M/(1+M^2),2/(1+M)\}\,.
\]
\end{rem}

\noindent Furthermore, as in \cite[Proposition 18]{CCL}, one can show that the solutions of \eqref{eq:eq_rho} and \eqref{eq:eq_Omega} can be extended also outside $J_{\delta_1}$ and $K_{\delta_2}$. In the sequel we will also make explicit the dependence of $\rho$ and $\Omega$ from $\alpha_0$.

\begin{cor}
\label{cor:extension_rhoOmega}
There exists $A_0>0$ such that if $0<\vert\alpha_0\vert< A_0$, then the solutions $\rho(p,\alpha_0)$, $\Omega(p,\alpha_0)$ obtained by Lemma \ref{lem:contraction} can be extended to the open left half-plane and are meromorphic on $\{p\in\C:\re(p)<0\}\setminus B(\omega)$.
\end{cor}

\begin{proof}
Observe that the coefficients in \eqref{eq:eq_rho} and \eqref{eq:eq_Omega} are analytic on $\{p\in\C:\re(p)\in[0,1]\}$, except for logarithmic branch points at $p=\imath\omega n, n\in\Z$. Then, using the explicit form of \eqref{eq:eq_rho} and \eqref{eq:eq_Omega}, one can extend $\rho(p,\alpha_0),\,\Omega(p,\alpha_0)$ to the whole $\{p\in\C:\re(p)\in[0,1]\}$ in a unique way and, in addition, prove that they are meromorphic here, up to the branch points.

Therefore, one can repeat the same argument of Lemma \ref{lem:contraction} and the previous remarks in order to extend in a unique way the solutions, first to $\{p\in\C:\re(p)\in[-1,0]\}$, and then (proceeding in an iterative way) to the whole $\{p\in\C:\re(p)<0\}$. In addition, they are meromorphic here, up to the branch cuts.
\end{proof}

Now, we can go back to the study of \eqref{eq:homogeneous_eq} on $\S_\ell(\omega)$. However, before presenting the proof of Theorem \ref{teo:poles_smallalpha}, a further auxiliary result is necessary. In particular, we prove that, for $\alpha_0$ small, points where \eqref{eq:homogeneous_eq} admits a non trivial solution are distant more than $C\alpha^2_0$ from the point $p_{\mathrm{s}}$ defined by \eqref{eq-truepole}, for some $C>0$.

\begin{pro}
\label{pro:distance_poles}
Under the assumptions of Theorem \ref{teo:poles_smallalpha}, there exist $C>0$ and $A_0>0$ such that, for $0<|\alpha_0|<A_0$, \eqref{eq:homogeneous_eq} admits only the trivial solution for every $p\in\{p\in\S_\ell(\omega):|p-p_{\mathrm{s}}|\geq C\alpha^2_0\}$.
\end{pro}

\begin{proof}
Preliminarily, note that, if $p\in\S_\ell(\omega)$ supports a solution $\widehat{\xi}(p)$ of \eqref{eq:homogeneous_eq}, then $\widehat{\xi}(p)$ solves as well
\begin{equation}
 \label{eq-L^2}
 \widehat{\xi}(p)=\cL^2(\alpha_0,p)\widehat{\xi}(p).
\end{equation}
Therefore, let us focus on this new equation. Fix a generic point $p\in\S_\ell(\omega)$ and define $z$ so that $p=p_{\mathrm{s}}+\alpha^2_0 z$.  An easy computation yields
\begin{multline}
\label{eq:L^2}
(\cL^2(\alpha_0,p)\widehat{\xi}(p))_n=\alpha^2_0h(p+\imath\omega n)
\left\{h\big(p+\imath\omega(n+1)\big)\widehat{\xi}_{n+2}(p)+\right.\\
\left.-\left[h\big(p+\imath\omega(n+1)\big)+h\big(p+\imath\omega(n-1)\big)\right]\widehat{\xi}_n(p))+h\big(p+\imath\omega(n-1)\big)\widehat{\xi}_{n-2}(p)\right\},
\end{multline}
for every $n\in\Z$. Now, by Taylor expansion, we have that, for $|\alpha_0|$ small,
\begin{equation}
\label{eq:log-expansion}
\log(p_{\mathrm{s}}+\alpha^2_0 z+\imath\omega n)=\log(p_{\mathrm{s}}+\imath\omega n)+\frac{\alpha_0^2z}{p_{\mathrm{s}}+\imath\omega n}+\mathcal{O}(\alpha_0^4),\qquad\forall n\in\Z\,.
\end{equation} 
Note that \eqref{eq:nonres} guarantees that $p_{\mathrm{s}}+\imath\omega n\neq0$, so that above expansion actually makes sense.
As a consequence,
\begin{multline}
\label{eq:h-expansion}
\displaystyle
\alpha_0^2h\big(p_{\mathrm{s}}+\alpha_0^2z+\imath\omega n\big)h\big(p_{\mathrm{s}}+\alpha_0^2z+\imath\omega (n+1)\big)\\
=\frac{-4\pi^2\alpha_0^2}{\bigg(\log\left(\frac{p_{\mathrm{s}}+\imath\omega n}{p_{\mathrm{s}}}\right)+\frac{\alpha_0^2z}{p_{\mathrm{s}}+\imath\omega n}+\mathcal{O}(\alpha_0^4)\bigg)\bigg(\log\left(\frac{p_{\mathrm{s}}+\imath\omega (n+1)}{p_{\mathrm{s}}}\right)+\frac{\alpha_0^2z}{p_{\mathrm{s}}+\imath\omega (n+1)}+\mathcal{O}(\alpha_0^4)\bigg)}.
\end{multline}
Hence, after some computations, one can check that there exists a constant $C>0$ such that, if $|z|>C$ and $\alpha_0$ is sufficiently small, then
\begin{equation}\label{eq:coeff}
 |\alpha_0^2h\big(p_{\mathrm{s}}+\alpha_0^2z+\imath\omega n\big)h\big(p_{\mathrm{s}}+\alpha_0^2z+\imath\omega (n+1)\big)|\leq1/8,\qquad\forall n\in\Z\,.
\end{equation}
In particular, the lower bound on $\vert z\vert$ is needed to deal with the case $n=0$, in order to control the denominator in the r.h.s. of \eqref{eq:coeff}.
whence, combining with \eqref{eq:L^2}, there results that
\[
\Vert \cL^2(\alpha_0,p_{\mathrm{s}}+\alpha^2_0z)\Vert_{\ell^2\to\ell^2}<\frac{1}{2}.
\]
Finally, as this entails $\text{I}-\cL^2(\alpha_0,p_{\mathrm{s}}+\alpha^2_0z)$ to be invertible, equation \eqref{eq-L^2} has the sole trivial solution for every $p\in\{p\in\S_\ell(\omega):|p-p_{\mathrm{s}}|\geq C\alpha^2_0\}$ and the same holds for \eqref{eq:homogeneous_eq}.
\end{proof}
\begin{rem}
Observe that, as shown by the previous proof, the no-resonance assumption \eqref{eq:nonres} is used also in the study of poles of $\widehat q$ and not only in the analysis of branch points on the imaginary axis.
\end{rem}

\noindent We have now all the tools required for the proof of Theorem \ref{teo:poles_smallalpha}.

\begin{proof}[Proof of Theorem \ref{teo:poles_smallalpha}]
Assume that $p\in\S_\ell(\omega)$ is a pole of $(\ID-\cL(\alpha_0,p))^{-1}$, namely that \eqref{eq:homogeneous_eq} admits a nontrivial solution at $p$. By Proposition \ref{pro:distance_poles}, we have that
\begin{equation}
 \label{eq-decomposition}
 p=p_{\mathrm{s}}+\alpha_0^2z,\qquad\text{with}\quad|z|\leq C,
\end{equation}
for some suitable $C>0$. Moreover, we know that $\rho(p,\alpha_0),\,\Omega(p,\alpha_0)$, defined by \eqref{eq:rho}-\eqref{eq:Omega}, are well-defined and non vanishing, and solve \eqref{eq:eq_rho}-\eqref{eq:eq_Omega}, respectively.

On the other hand, by definition, these two functions have to satisfy
\begin{equation}
\label{eq:poles-equation}
\displaystyle
\frac{1}{\rho(p,\alpha_0)}-\Omega(p,\alpha_0)=0.
\end{equation}
However, Lemma \ref{lem:contraction}, Corollary \ref{cor:extension_rhoOmega} and \eqref{eq:rho}, \eqref{eq:Omega} guarantee that the solutions of \eqref{eq:eq_rho}-\eqref{eq:eq_Omega} are uniquely determined and have some specific features. In the following, we discuss at which points of the form \eqref{eq-decomposition}, these functions may satisfy \eqref{eq:poles-equation}.

Define, therefore, the function 
\[
F(z,\alpha_0):=\frac{1}{\alpha_0}\left(\frac{1}{\rho(p_{\mathrm{s}}+\alpha^2_0z,\alpha_0)}-\Omega(p_{\mathrm{s}}+\alpha^2_0z,\alpha_0)\right),\qquad \vert z\vert \leq C, \quad0<\vert\alpha_0\vert<\delta\ll1,
\]
so that \eqref{eq:poles-equation} is equivalent to
 \begin{equation}
 \label{eq-equiv}
 F(z,\alpha_0)=0.
 \end{equation}
Using \eqref{eq:eq_rho} and \eqref{eq:eq_Omega} and the definition of $p_{\mathrm{s}}$, we find
\[
\rho(p_{\mathrm{s}}+\alpha^2_0z)=\alpha_0h(p_{\mathrm{s}}+\alpha^2_0z)+\mathcal{O}(\alpha_0^3)=\frac{2\pi\imath}{\frac{\alpha_0z}{p_{\mathrm{s}}}+\mathcal{O}(\alpha^3_0)}+\mathcal{O}(\alpha_0^3),
\]
so that
\[
\frac{1}{\rho(p_{\mathrm{s}}+\alpha^2_0z,\alpha_0)}=\frac{\alpha_0z}{2p_{\mathrm{s}}\pi\imath}+\mathcal{O}(\alpha^3_0),
\]
and
\[
\begin{split}
\Omega(p_{\mathrm{s}}+\alpha^2_0z,\alpha_0)&=-\alpha_0h(p_{\mathrm{s}}+\alpha^2_0z-\imath\omega)+\mathcal{O}(\alpha^3_0) \\
& = - \frac{2\pi\imath\alpha_0}{\log(p_{\mathrm{s}}-\imath\omega)+4\pi\beta_0+\mathcal{O}(\alpha^2_0)}+\mathcal{O}(\alpha^3_0).
\end{split}
\]
As a consequence
\begin{equation}
\label{eq:F-expansion}
F(z,\alpha_0)=\frac{z}{2p_{\mathrm{s}}\pi\imath}+\frac{2\pi\imath}{\log(p_{\mathrm{s}}-\imath\omega)+4\pi\beta_0+\mathcal{O}(\alpha^2_0)}+\mathcal{O}(\alpha_0^2).
\end{equation}
Thus the unique solution of \eqref{eq-equiv} for $\alpha_0=0$ is given by
\[
 \overline{z}=\left\{
 \begin{array}{ll}
 \displaystyle \frac{4\imath\pi^2 e^{2(\log2-\gamma)}}{\log\big(e^{2(\log2-\gamma)}-\omega\big)-2(\log2-\gamma)}, & \text{if}\quad \omega<e^{2(\log2-\gamma)},\\[.6cm]
 \displaystyle -\frac{2\pi^3e^{2(\log2-\gamma)}}{\pi^2+\Lambda(\omega)^2}+\imath\frac{2\pi^2\Lambda(\omega)e^{2(\log2-\gamma)}}{\pi^2+\Lambda(\omega)^2}, & \text{if}\quad \omega>e^{2(\log2-\gamma)},
 \end{array}
 \right.
\]
where
\begin{equation}
\Lambda(\omega):=\log(\omega-e^{2(\log2-\gamma)})-2(\log2-\gamma).
\end{equation} 
Now, as $F(\overline{z},0)=0$ and $\displaystyle \frac{\partial F}{\partial z}(\overline{z},0)=\frac{1}{2p_{\mathrm{s}}\pi\imath}\neq0$, the (analytic) implicit function theorem ensures that for some $\delta_0>0$ there exists an analytic function $(-\delta_0,\delta_0)\ni\alpha_0\mapsto z(\alpha_0)\in\IN{\S_\ell(\omega)}$ such that $z(0)=\overline{z}$ and
\[
F(z(\alpha_0),\alpha_0)=0,\qquad \forall\,0<\vert\alpha_0\vert<\delta_0.
\]
Then, claim follows letting $p_0:=p_{\mathrm{s}}+\alpha^2_0z(\alpha_0)$, for $\alpha_0\in(-\delta_0,\delta_0)\setminus\{0\}$.
\end{proof}


\subsubsection{Poles of $(\ID-\mathcal{L}(p))^{-1}$ for arbitrary $\alpha_0\neq0$}\label{subsec:poleslargealpha}

It is left to show that $(\ID-\mathcal{L}(\alpha_0,p))^{-1}$ possesses at most one simple pole in $\S_\ell(\omega)$, belonging to $\IN{\S_\ell(\omega)}$, for every $\alpha_0\neq0$.

Preliminarily, we establish the following result (whose immediate proof is omitted and retraces exactly that of \cite[Proposition 24]{CCL}).

\begin{pro}
\label{pro:neumann_series}
Assume that \eqref{eq:nonres} is satisfied. For every $A_0>0$ there exists $Q>0$ such that $(\ID-\cL(p,\alpha_0))^{-1}$ is invertible in $\ell^2(\Z)$ and
\[
(\ID-\cL(p,\alpha_0))^{-1}=\sum_{m\geq0}\cL^m(p,\alpha_0),
\] 
for all $(p,\alpha_0)\in\{(p,\alpha_0)\in\S_\ell(\omega)\times\mathbb{R}:0<\vert \alpha_0\vert<A_0,\:\re(p)<-Q\}$. In addition, as an operator valued function, $(\ID-\cL(p,\alpha_0))^{-1}$ is analytic on $\{(p,\alpha_0)\in\IN{\S_\ell(\omega)}\times\mathbb{R}: \alpha_0\in(-A_0,A_0),\:\re(p)<-Q\}$.
\end{pro}

It is also necessary to introduce a second auxiliary result. First note that the definition of $\cL(\alpha_0,p)$ given by \eqref{operator} can be extended to the whole complex plane. However, the function $p\mapsto\cL(\alpha_0,p)$ presents jump discontinuities at branch cuts in $B(\omega)$.

In addition, let $T:\ell^2(\Z)\to\ell^2(\Z)$ be the shift operator defined by
\[
 (Tu)_n:=u_{n-1},\quad\forall n\in\Z,\qquad\forall u=(u_n)_n\in\ell_2(\Z).
\]
A straightforward computation shows that 
\begin{equation}
\label{eq-period}
\mathcal{L}(\alpha_0,p+\imath\omega)=T^{-1}\mathcal{L}(\alpha_0,p)T,
\end{equation}
which allows to prove the following property.

\begin{lem}
\label{lem:periodic_operator}
Let $p\mapsto\mathcal{L}_2(\alpha_0,p)$ be the $\ell_2(\Z)$-operator valued function defined by 
\begin{equation}\label{eq:L2}
\mathcal{L}_2(\alpha_0,p):=T^{\frac{p}{\imath\omega}}\mathcal{L}(\alpha_0,p) T^{-\frac{p}{\imath\omega}}.
\end{equation}
Then, $\mathcal{L}_2(\alpha_0,p+\imath\omega)=\mathcal{L}_2(\alpha_0,p)$ and $\sigma(\mathcal{L}_2(\alpha_0,p))=\sigma(\mathcal{L}(\alpha_0,p))$, for every $p\in\C$.
\end{lem}

\begin{proof}
 The proof is the same of \cite[Lemma 21]{CCL}. We show it for the sake of completeness. As $T$ is unitary, $T^{\frac{p}{\imath\omega}}$ and $T^{-\frac{p}{\imath\omega}}$ are well defined by holomorphic functional calculus. Thus the operators $\mathcal L_2(\alpha_0,p)$ form an holomorphic family, with respect to the parameter $p$. In addition, a straightforward computation yields the $\imath\omega$-periodicity. Finally, isospectrality descends by the fact that the spectrum of an operator is invariant under conjugation by invertible operators (as conjugation preserves spectral measures).
\end{proof}

\begin{rem}
 Note that, the $\imath\omega$-periodicity implies that $\cL_2(\alpha_0,\cdot)$ does not present discontinuities at $B(\omega)$. Moreover, isospectrality implies that poles of $(\ID-\cL(\alpha_0,p))^{-1}$ coincides with those of $(\ID-\cL_2(\alpha_0,p))^{-1}$.
\end{rem}

Now, we have all the ingredients to determine the location of possible poles for arbitrary $\alpha_0\neq0$. The proof is analogous to that of \cite[Proof of Theorem 19]{CCL}. We report it for the sake of completeness.

\begin{teo}
\label{teo:poles_nonzeroalpha}
Assume that \eqref{eq:nonres} is satisfied and let $\alpha_0\in\R\setminus\{0\}$. Then, $(\ID-\mathcal{L}(\alpha_0,p))^{-1}$ presents at most one pole $p_0=p_0(\alpha_0,\omega)$ on $\S_\ell(\omega)$. In particular, if such $p_0$ does exists, then it actually belongs to $\IN{\S_\ell(\omega)}$. As a consequence, $\widehat{q}$ can be extended to the open left half-plane, up to at most a sequence of poles defined as in \eqref{eq-expansion}.
\end{teo}

\begin{proof}
From Lemma \ref{lem:periodic_operator}, it suffices to prove the claim for $(\text{I}-\cL_2(\alpha_0,p))^{-1}$. Fix $p$ in the open left half-plane. Being conjugated to $\cL(\alpha_0,p)$, $\cL_2(\alpha_0,p)$ is a compact operator on $\ell^2(\Z)$. Then, there exist a sequence $(F_k(p,\alpha_0))_k$ of finite rank operators converging in the $\ell^2(\Z)$-operator topology to $\cL_2$, and a sequence $(P_k)_k$ the projectors on the range of $F_k(p,\alpha_0)$. Notice that a finite rank approximation of $\mathcal L$ can be obtained as in the proof of Prop. \ref{profondo rosso}. Then $F_k$ is obtained by conjugation as in \eqref{eq:L2}.

Now, let $A_0>0$ and take $Q=Q(\alpha_0)$ as in Proposition \ref{pro:neumann_series}, so that $(I-\cL_2(p,\alpha_0))$ is invertible in $\{(p,\alpha_0)\in\S_\ell(\omega)\times\mathbb{R}:0<\vert \alpha_0\vert<A_0,\:\re(p)<-Q\}$. Now, take $\varepsilon>0$ such that there exists $|\alpha_0'|$ so that the possible pole provided by Theorem \ref{teo:poles_smallalpha} has real part smaller that $-\varepsilon$. Take also $k_0=k_0(A_0)$ such that for $|\alpha_0|\in(0,A_0)$ and $p\in \S_b$, with
\[
\S_b(\varepsilon,\omega):=\{z\in\C:\re(z)\in(-Q,-\varepsilon),\,\im(p)\in[0,\omega]\}\subset\S_\ell(\omega),
\]
there holds $\Vert \cL_2(\alpha_0,p) - F_k(p,\alpha_0)\Vert_{\ell_2(\Z)\to\ell_2(\Z)}<\varepsilon$, for every $k>k_0$. The identity,
\[
(\ID-\cL_2(\alpha_0,p))=(\ID-F(p,\alpha_0))(\ID-(\cL_2(\alpha_0,p)-F_k(\alpha_0,p))),
\]
with
\begin{equation}\label{eq:F}
 F(\alpha_0,p):=F_k(\alpha_0,p)[\ID-(\cL_2(\alpha_0,p)-F_k(\alpha_0,p))]^{-1}
\end{equation}
 gives that $\ID-\cL_2(\alpha_0,p)$ is invertible if and only if $\ID-F(\alpha_0,p)$ is invertible. Observe that $F(\alpha_0,p)$ is of finite rank and the Fredholm alternative implies that $\ID-F(\alpha_0,p)$ is not invertible if and only if $u=F(\alpha_0,p)u$ has a non trivial solution. Moreover, since $F(\alpha_0,p)=P_kF(\alpha_0,p)$, if there exists a non trivial $u$ such that $u=F(\alpha_0,p)u$, then $u=P_ku$. Thus, $(\ID-\cL_2(\alpha_0,p))^{-1}$ has a pole in $\S_b(\varepsilon,\omega)$ if and only if 
 \begin{equation}\label{eq:f}
 f(\alpha_0,p):=\det(P_k(\ID-F(\alpha_0,p))P_k)=0\,.
 \end{equation}

We know from Theorem \ref{teo:poles_smallalpha} that, for small $|\alpha_0|$, $f$ may have one simple zero in $\IN{\S_b(\varepsilon,\omega)}$. Therefore,
\begin{equation}
\label{eq-Z}
 Z(\alpha_0):=\frac{1}{2\pi\imath}\int_{\partial^+\S_b(\varepsilon,\omega)}\frac{\partial_pf(\alpha_0,p)}{f(\alpha_0,p)}\dP
\end{equation}
is at most equal to 1 for small $|\alpha_0|$. In order to conclude, it is sufficient to prove that $Z$ is at most 1 for every $|\alpha_0|\in(0,A_0)$.

However, arguing as in \cite[Lemma 20]{CCL} and using the $\imath\omega$-periodicity of $\cL_2(\alpha_0,p)$, the previous integral can be restricted to the vertical sides. Moreover, Proposition \ref{pro:neumann_series} implies that, if one moves leftwards the left vertical side of $\partial\S_b(\varepsilon,\omega)$, then the integral in \eqref{eq-Z} does not change. In addition, the contribution of this side vanishes as $\re(p)\to-\infty$. Hence,
\[
 Z(\alpha_0)=\frac{1}{2\pi\imath}\int_{-\varepsilon}^{-\varepsilon+\imath\omega}\frac{\partial_pf(\alpha_0,p)}{f(\alpha_0,p)}\dP.
\]
Observe that $Z$ is continuous in $\alpha_0\in\R$. Indeed, the coefficient in $\mathcal L(p,\alpha_0)$ depends linearly on $\alpha_0$, so that such operator is continuous in $\alpha_0$, as well as $\mathcal L_2$. Then the finite-rank approximations $F_k(\alpha_0,p)$ of the latter are also continuous in $\alpha_0$. Expanding in Neumann series the term $[\ID-(\cL_2(\alpha_0,p)-F_k(\alpha_0,p))]^{-1}$ in \eqref{eq:F}, one sees that $F(\alpha_0, p)$ is also continuous in $\alpha_0$. Consequently, the same holds for $f(\alpha_0,p)$ in \eqref{eq:f}. Moreover, by the previous discussion $f(\alpha_0,p)$ is given by the sum of terms where the dependence on $\alpha_0$ and $p$ is factorized, so that $\partial_p f(\alpha_0,p)$ is continuous in $\alpha_0$.
Finally, since $Z$ takes value in $\N$, it has to be constant, thus completing the proof.
\end{proof}

\begin{rem}
 We mention again that the extension of $\widehat{q}$ to the open left half-plane obtained in Theorem \ref{teo:poles_nonzeroalpha} is not analytic here up to the poles. Indeed, in $B(\omega)$ it presents jumps due to the logarithmic branch cuts. However, it does not present poles there.
\end{rem}

\begin{rem}
 \label{rem-monocrom}
 It is not clear, at the moment, how to adapt the whole technique developed and used in Section \ref{subsec:lefthalfplane} to multi-chromatic point perturbations. This motivates that the statements of Theorem \ref{teo:cion} presents the monochromatic assumption.
\end{rem}

\begin{rem}
 \label{rem-whynonres_bis}
 We also mention that the no-resonance assumption is used not only in Section \ref{subsec:inversion}, but also in Section \ref{subsec:lefthalfplane}, in analogy with the analysis performed in \cite{CCL} for the one-dimensional case.
\end{rem}


\section{Complete ionization: proof of Theorem \ref{teo:cion}}
\label{sec:ionization}

In this final section we present the proof Theorem \ref{teo:cion}, which is the main goal of the paper. Recalling \eqref{eq:survival} and \eqref{eq:theta} it is sufficient to study the decay of $Z$.

To this aim we will use information obtained in Section \ref{sec:ion}, to estimate $Z$ by its Laplace inverse transform, that is
\begin{equation}
\label{inverselap}
Z(t)=\frac{1}{2\pi\imath}\lim_{R\to\infty}\int_{-\imath\, R}^{\imath\, R}e^{pt}\widehat{Z}(p)\dP,
\end{equation}
where $\widehat{Z}$ denotes the Laplace transform of $Z$.

Before dealing with \eqref{inverselap}, a further preliminary result is needed.

\begin{lem}\label{lem:qinfinity}
Under the assumptions of Theorem \ref{teo:cion}, for every $a\in\R$ there holds
\begin{equation}
\label{eq:qinfinity}
\lim_{\tau\to+\infty}\widehat{q}(\imath a-\tau)=0
\end{equation}
\end{lem}

\begin{proof}
As in the previous section, it is convenient to use the $\ell_2$ sequence form for the function $\widehat{q}$ and focus on $a\in[0,\omega]$. Then, given $a\in[0,\omega]$, from Theorem \ref{teo:poles_nonzeroalpha}, for every $\tau>0$ (up to at most one point) there results
\[
\widehat{q}(\imath a-\tau)=(I-\cL(\imath a-\tau))^{-1}\widehat{g}(\imath a-\tau).
\]
One sees that $\|\cL(\imath a-\tau)\|_{\ell_2(\Z)\to\ell_2(\Z)}\to0$, as $\tau\to+\infty$, so that $\Vert (I-\cL(\imath a-\tau))^{-1}\Vert_{\ell_2(\Z)\to\ell_2(\Z)}$ is uniformly bounded for large $\tau$. Moreover, is is easy to see that also $\Vert \widehat{g}(\imath a-\tau)\Vert_{\ell^2(\Z)}$ is uniformly bounded for $\tau$ large enough.
Hence
\[
\Vert \widehat{q}(\imath a-\tau)\Vert^2_{\ell^2(\Z)}:=\sum_{n\in\Z}\vert \widehat{q}_n(\imath a-\tau)\vert^2\leq M,\qquad \forall \tau\geq T,
\]
for some suitable $M,\,T>0$, so that (as a single function)
\begin{equation}
 \label{eq-qbound}
 \vert \widehat{q}(\imath a-\tau+\imath \omega n)\vert\leq M,\qquad\forall n\in\Z,\quad\forall\tau>T\,.
\end{equation}
 Finally, \eqref{eq:qinfinity} follows from \eqref{suspiria} in view of \eqref{eq-qbound}.
\end{proof}

\begin{rem}
 Clearly, at the branch cuts \eqref{eq:qinfinity} has to be meant as valid for both the approximation from above and for the approximation from below.
\end{rem}

\begin{figure}[htbp]
        \centering
        \includegraphics[width=14cm]{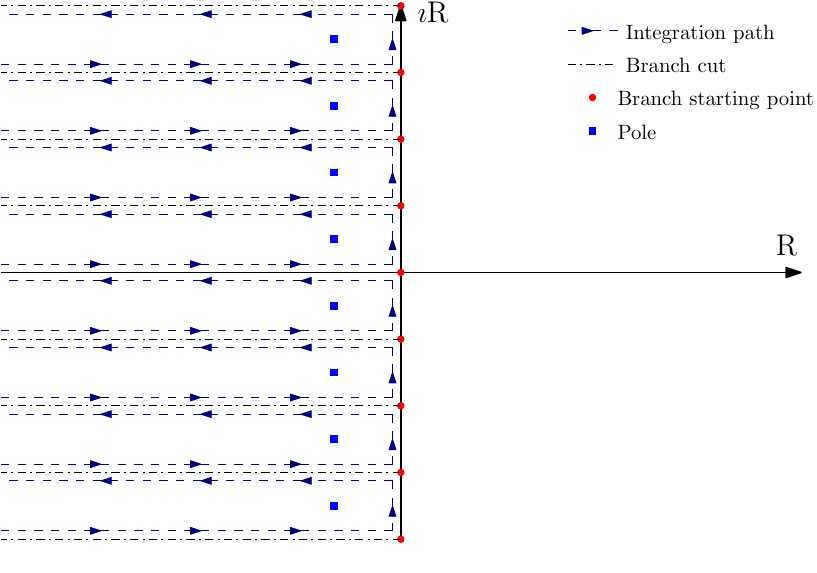}
        \caption{Dashed lines correspond the integration path. Branch points are indicated with red dots, while blue square dots represent simple poles.}
        \label{path}
\end{figure}

\begin{proof}[Proof of Theorem \ref{teo:cion}]
We choose the integration path depicted in Figure \ref{path} (as in \cite{CDFM,CCLR}). Using \eqref{eq:zeta_tr}, the properties of $\widehat{q}$ established in Sections \ref{subsec:inversion}-\ref{subsec:lefthalfplane} and Lemma \ref{eq:qinfinity}, the right-hand side of \eqref{inverselap} can be rewritten as 
\begin{equation}\label{eq:decayingterms}
\begin{split}
&\lim_{R\to\infty}\int_{-\imath\, R}^{\imath\, R}e^{p\,t}\widehat{Z}(p)\,dp=\underbrace{2\pi\imath\sum_{n\in\Z}\widehat{Z}_2(p_n)R_ne^{p_nt}}_{=:Y_1}+ \\
&+\underbrace{\sum_{n\in\Z}e^{\imath \omega nt}\int^{+\infty}_0 e^{-\tau\,t}[\widehat{Z}^+_2(-\tau+\imath\omega n)\widehat{q}^+(-\tau+\imath\omega n)-\widehat{Z}^-_2(-\tau+\imath\omega n)\widehat{q}^-(-\tau+\imath\omega n)]\,d\tau}_{=:Y_2}
\end{split}
\end{equation}
where $\beta_n=\imath\omega n$ are the branch point obtained in Section \ref{subsec:inversion},
\begin{gather*}
 \widehat{q}^\pm(-\tau+\imath\omega n):=\lim_{\varepsilon\to0}\widehat{q}(-\tau+\imath\omega n\pm\imath\varepsilon)\\
 \widehat{Z}^\pm_2(-\tau+\imath\omega n):=\lim_{\varepsilon\to0}\widehat{Z}_2(-\tau+\imath\omega n\pm\imath\varepsilon),
\end{gather*}
and $R_n:=\text{Res}_{p_n}(\widehat{q})$ are the residues at the possible poles $(p_n)_n$. Note that $Y_1$ is the contribution of the possible poles of $\widehat{q}(p)$, while $Y_2$ accounts for the branch cuts. We estimate the two terms separately.

\emph{Step (i): proof of $Y_2\sim \f{C_1+C_2\log t}{t}$, with $C_1,\,C_2\in\C$, as $t\to+\infty$}. To this aim, as a preliminary step, it is necessary to prove that 
\begin{equation}
\label{eq:sup_bounded} 
\sup_{n\in\Z\setminus\{0\}, \tau>0}\vert n\,\widehat{q}^\pm(\imath\omega n-\tau) \vert<\infty.
\end{equation}
We use again $\widehat{q}$ as a sequence of functions. From the previous section, we know that even if $\re(p)=-\tau<0$ and $\im(p)=0,\omega$, $\widehat{q}(p)=(\ID-\mathcal{L}(p))^{-1}\widehat{g}(p)$. Clearly, $\widehat{q}$ is not analytic here, but, by analytic continuation, $\widehat{q}(-\tau)$ and $\widehat{q}(-\tau+\imath\omega)$ are real analytic in $\tau$ and go to zero as $\tau\to+\infty$. In the sequel we focus on $\widehat{q}(-\tau)$, but for $\widehat{q}(-\tau+\imath\omega)$ the proof is analogous. Let, then, $n_0\in\N$ (a similar argument works for $n_0<0$) be such that $\sup_{n\geq n_0,\tau>0}\vert h(-\tau+\imath\omega n) \vert<1/2$. Note also that, by \eqref{eq-qbound}, $\sup_{\tau>0,n\in\Z}\vert \widehat{q}_n(-\tau) \vert\leq C<\infty$. Now, denoting by $P$ be the projector onto $\ell^2(n_0+\N)$, that is, on the sequences indexed from $n_0$ to $+\infty$, $P\widehat{q}$ has to solve
\begin{equation}
\label{eq:projected_eq}
P\widehat{q}(-\tau)=P\mathcal{L}(-\tau)P\widehat{q}(-\tau)+P\widehat{g}(-\tau)+R(-\tau),
\end{equation}
where the only non-zero component of $R(-\tau)$ is $R_{n_0}(-\tau)=\alpha_0h(-\tau+\beta_{n_0-1})\widehat{q}_{n_0-1}(-\tau)$. Therefore, to conclude the proof, it is sufficient to prove that \eqref{eq:projected_eq} has a unique solution in
\[
\widetilde{\ell}_\infty:=\{a=(a_n)_{n\geq n_0}: \Vert a \Vert_{\widetilde{\ell}_\infty}=\sup_{n\geq n_0}n\vert a_n\vert<\infty\}.
\]
First we observe that $g(-\tau)\in\widetilde{\ell}_\infty$, as can be easily seen from \eqref{operator}, and that $R(-\tau)\in\widetilde{\ell}_\infty$, as it has only one non-zero component. In addition, if $n_0$ is sufficiently large, then $\Vert P\mathcal{L}(-\tau)\Vert_{\widetilde{\ell}_\infty\to\widetilde{\ell}_\infty}<1/2$. Hence, \eqref{eq:projected_eq} is well-defined on $\widetilde{\ell}_\infty$ and $(\ID-P\mathcal{L}(-\tau))$ is invertible there, so that the equation admits a unique solution $(\ID-P\mathcal{L}(-\tau))^{-1}(P\widehat{g}(-\tau)+R(-\tau))\in \widetilde{\ell}_\infty\subset \ell^2(\Z)$. By uniqueness, such solution must coincide with $\widehat{q}(-\tau)$ and, thus, \eqref{eq:sup_bounded} follows.

We can now estimate $Y_2$. The function $\widehat{Z}_2(p)$ defined by \eqref{eq:laplaceZ2} is continuous across the half-lines $=-\tau+\imath\omega n$, $\tau\geq0$, for all $n\in\Z\setminus\{0\}$, while it presents a discontinuity across $(-\infty,0)$. Preliminarily, we can rewrite $Y_2$ as follows
\begin{multline*}
 Y_2=\underbrace{\int^{+\infty}_0 e^{-\tau\,t}[\widehat{Z}^+_2(-\tau)\widehat{q}^+(-\tau)-\widehat{Z}^-_2(-\tau)\widehat{q}^-(-\tau)]\,d\tau}_{=:I}+\\[.2cm]
 \underbrace{\sum_{n\in\Z\setminus\{0\}}e^{\imath \omega nt}\int^{+\infty}_0 e^{-\tau\,t}[\widehat{Z}^+_2(-\tau+\imath\omega n)\widehat{q}^+(-\tau+\imath\omega n)-\widehat{Z}^-_2(-\tau+\imath\omega n)\widehat{q}^-(-\tau+\imath\omega n)]\,d\tau}_{=:II}
\end{multline*}
Let us focus first on $II$. Combining this fact with \eqref{eq:sup_bounded}, we find that 
\[
|II|\lesssim  \int^{+\infty}_0 e^{-\tau\,t}\sum_{n\in\Z\setminus\{0\}}\frac{\vert\widehat{Z}_2(-\tau+\imath\omega n)\vert}{|n|}\,d\tau \\ 
\]
By \eqref{eq:laplaceZ2}, for $|n|$ large,
\[
\vert \widehat{Z}_2(-\tau+\imath\omega n)\vert\lesssim\frac{\log (\sqrt{\tau^2+n^2\omega^2})+\pi+\vert\log(-\lambda_0)\vert}{\sqrt{\tau^2+(n\omega+\lambda_0)^2}},\qquad \forall\tau>0.
\]
Furthermore, for $ n$ large, $\displaystyle \vert n\omega+\lambda_0\vert> \frac{\vert n\vert\omega}{2}$, so that $\tau^2+(n\omega+\lambda_0)^2>\frac{1}{4}(\tau^2+n^2\omega^2)$.  Thus,
\[
\vert \widehat{Z}_2(-\tau+\imath\omega n)\vert\lesssim\frac{\log (\sqrt{\tau^2+n^2\omega^2})+\pi+\vert\log(-\lambda_0)\vert}{\sqrt{\tau^2+n^2\omega^2}}=:f_n(\tau)\,,\qquad \forall\tau>0.
\]
Now, a straightforward computations yields 
\[
f'_n(\tau)=\frac{\tau}{(\tau^2+n^2+\omega^2)^{3/2}}(1-\log(\sqrt{\tau^2+n^2\omega^2})-\pi-\vert\log(-\lambda_0)\vert)
\]
so that
\[
f'_n(\tau)<0 \iff \log(\sqrt{\tau^2+n^2\omega^2})+\pi+\vert \log(-\lambda_0)\vert>1,
\]
and then, for $n$ large, $f_n(\cdot)$ is strictly decreasing. Thus,
 \[
 f_n(\tau)\leq f_n(0)=\frac{\log(|n|\omega)+\pi+\vert\log(-\lambda_0)\vert}{2\vert n\vert\omega},\qquad\forall \tau>0.
 \]
Consequently,
\[
\frac{\vert \widehat{Z}_2(-\tau+\imath\omega n)\vert}{|n|}\lesssim \frac{\log \vert n\vert}{n^2},\qquad\forall \tau>0
\]
which entails that $|II|\lesssim 1/t$, for every $t>0$. 
\medskip

It is then left to deal with $I$. Setting
\[
\log^-(-\tau):=\lim_{\varepsilon\to0^+}\log(-\tau-\imath\varepsilon)=\log(\tau)-\imath\pi,
\]
and
\[
\log^+(-\tau):=\lim_{\varepsilon\to0^+}\log(-\tau+\imath\varepsilon)=\log(\tau)+\imath\pi,
\]
by \eqref{eq:laplaceZ2}, \eqref{ilgattoanovecode}, \eqref{eq-taumenouno}, \eqref{eq-erremenouno}, \eqref{eq-acca} and \eqref{eq-cu} one obtains that, for $\tau>0$,
\[
\widehat Z^+_2(-\tau)\widehat q^+(-\tau)=\f{\imath\const_0\left[\log\big(\f{\tau}{-\la_0}\big)+\imath\f{\pi}{2}\right]}{4\pi(-\tau+\imath\la_0)(Q(-\tau)+\log\tau+\imath\pi)}\left[H(-\tau)+\f{\const_0\left[\log\big(\f{\tau}{-\la_0}\big)+\imath\f{\pi}{2}\right]}{-\tau+\imath\la_0}\right]
\]
and
\begin{multline*}
\widehat Z^-_2(-\tau)\widehat q^-(-\tau)=\f{\imath\const_0\left[\log\big(\f{\tau}{-\la_0}\big)-\imath\f{3\pi}{2}\right]}{4\pi(-\tau+\imath\la_0)}\:\times\\[.2cm]
\times\:\left\{\widehat{R}_{-1}(-\tau)+\f{\widehat{\tau}_{-1}(-\tau)}{Q(-\tau)+\log\tau-\imath\pi}\left[H(-\tau)+\f{\const_0\left[\log\big(\f{\tau}{-\la_0}\big)-\imath\f{3\pi}{2}\right]}{-\tau+\imath\la_0}\right]\right\},
\end{multline*}
so that
\begin{multline*}
\widehat Z^+_2\big(-\tf{y}{t}\big)\widehat q^+\big(-\tf{y}{t}\big)-\widehat Z^-_2\big(-\tf{y}{t}\big)\widehat q^-\big(-\tf{y}{t}\big)=\\[.2cm]
=\big(a(t,y)-b(t,y)\big)H\big(-\tf{y}{t}\big)+a(t,y)\f{\const_0\log\big(\f{y}{-\la_0}\big)+\imath\const_0\f{\pi}{2}}{-\f{y}{t}+\imath\la_0}-b(t,y)\f{\const_0\log\big(\f{y}{-\la_0}\big)-\imath\const_0\f{3\pi}{2}}{-\f{y}{t}+\imath\la_0}\\[.2cm]
+\log t\bigg[-c(t,y)\widehat{R}_{-1}\big(-\tf{y}{t}\big)+\f{\const_0}{-\f{y}{t}+\imath\la_0}\big(b(t,y)-a(t,y)\big)\bigg],
\end{multline*}
with
\[
 \begin{split}
  a(t,y)&:=\f{\imath\const_0\bigg[\f{\log\big(\f{y}{-\la_0}\big)+\imath\f{\pi}{2}}{\log t}-1\bigg]}{4\pi\big(-\f{y}{t}+\imath\la_0\big)\left(\f{Q\big(-\f{y}{t}\big)+\log y+\imath\pi}{\log t}-1\right)},\\[.2cm]
  b(t,y)&:=\f{\imath\const_0\widehat{\tau}_{-1}\big(-\f{y}{t}\big)\bigg[\f{\log\big(\f{y}{-\la_0}\big)-\imath\f{3\pi}{2}}{\log t}-1\bigg]}{4\pi\big(-\f{y}{t}+\imath\la_0\big)\left(\f{Q\big(-\f{y}{t}\big)+\log y-\imath\pi}{\log t}-1\right)},\\[.2cm]
  c(t,y)&:=\f{\imath\const_0\widehat{\tau}_{-1}\big(-\f{y}{t}\big)\bigg[\f{\log\big(\f{y}{-\la_0}\big)-\imath\f{3\pi}{2}}{\log t}-1\bigg]}{4\pi\big(-\f{y}{t}+\imath\la_0\big)}.
 \end{split}
\]
Now, since for a.e. $y>0$
\[
 \lim_{t\to+\infty}a(t,y)=\f{\const_0}{4\pi\la_0},\qquad \lim_{t\to+\infty}b(t,y)=\f{\const_0\widehat{\tau}_{-1}(0)}{4\pi\la_0},\qquad \lim_{t\to+\infty}c(t,y)=-\f{\const_0\widehat{\tau}_{-1}(0)}{4\pi\la_0}
\]
and since $a(t,\cdot),\,b(t,\cdot),\,c(t,\cdot),\,\widehat{\tau}_{-1}\big(-\tf{(\cdot)}{t}\big),\,\widehat{R}_{-1}\big(-\tf{(\cdot)}{t}\big),\,H\big(-\tf{(\cdot)}{t}\big),\,Q\big(-\tf{(\cdot)}{t}\big)$ are equibounded in $t$ (for $a(t,\cdot),\,b(t,\cdot),\,c(t,\cdot)$ is a direct computation, while for $\widehat{\tau}_{-1}\big(-\tf{(\cdot)}{t}\big),\,\widehat{R}_{-1}\big(-\tf{(\cdot)}{t}\big),\,H\big(-\tf{(\cdot)}{t}\big),\,Q\big(-\tf{(\cdot)}{t}\big)$ one can use \eqref{eq-taumenouno}, \eqref{eq-erremenouno}, \eqref{eq-acca} and \eqref{eq-cu} and argue as in the proof of Lemma \ref{lem:qinfinity}), by dominated convergence one finds that
\begin{multline*}
I=\frac{1}{t}\int^{+\infty}_0e^{-y}\left[\widehat Z^+_2\big(-\tf{y}{t}\big)\widehat q^+\big(-\tf{y}{t}\big)-\widehat Z^-_2\big(-\tf{y}{t}\big)\widehat q^-\big(-\tf{y}{t}\big)\right]\,dy\,\\[.2cm]
\sim\f{C_1+C_2\log t}{t},\qquad\text{as}\quad t\to+\infty,
\end{multline*}
with
\begin{gather}
 \label{eq-ciuno}C_1:=\f{\const_0}{4\pi\la_0}\bigg[\big(1-\widehat{\tau}_{-1}(0)\big)\bigg(H(0)-\const_0\f{\gamma+\log(-\la_0)}{\imath\la_0}\bigg)+\f{\pi\const_0}{2\la_0}\big(1-3\widehat{\tau}_{-1}(0)\big)\bigg]\\[.2cm]
 \label{eq-cidue}C_2:=\widehat{\tau}_{-1}(0)\widehat{R}_{-1}(0)+\f{\const_0}{\imath\la_0}\big(\widehat{\tau}_{-1}(0)-1\big).
 \end{gather}
Thus, the proof of the asymptotics is complete up to possibly redefining $C_1$.

\bigskip

\emph{Step (ii): proof of $|Y_1|\lesssim e^{-bt}$, for every $t>0$, for some $b>0$}. It is sufficient to prove that, the sequence of residues $(R_n)_n$, has a sufficiently fast decay for $n$ large, so that $(\widehat{Z}(p_n)R_n)_n\in\ell^1(\Z)$. Indeed, if this is the case, then the desired estimate follows since $\re(p_n)\leq -C_{\alpha_0}<0$.

Let $r>0$ be such that $B_{n,r}:=\{\vert p-p_n\vert\leq r\}\subset\subset\IN{\S_\ell(\omega)}$. Integrating \eqref{Poe} term by term on $\partial^+B_{n,r}$ and recalling that each $\widehat{g}_n$ is analytic on $B_{n,r}$, we obtain
\begin{equation}
\label{eq:residue_equation}
R_n=\cL(p_n)[R_{n+1}-R_{n-1}].
\end{equation}
Now, arguing as for \eqref{eq:sup_bounded}, we find that the residues satisfy
\begin{equation}
\label{eq:residues_decay}
\sup_{n\in \mathbb{Z}\setminus\{0\}}\vert n R_n \vert<+\infty\, \iff \vert R_n\vert\lesssim \frac{1}{\vert n\vert} \qquad \forall n\in\mathbb{Z}\setminus\{ 0\}\,.
\end{equation}
Hence, as by \eqref{eq:laplaceZ2},
\[
\vert Z_2(p_n)\vert\lesssim \frac{\log |n|}{|n|},\qquad\text{for}\quad |n|\quad\text{large},
\]
we get $(\widehat{Z}(p_n)R_n)_n\in\ell^1(\Z)$.
\end{proof}

\begin{rem}
 \label{rem-cost-zero}
 Note that, looking at \eqref{eq-ciuno} and \eqref{eq-cidue}, one can immediately see that it is not possible to exclude in general the vanishing of the two constants of the asymptotics. However, it is clear as well that such a vanishing may be, in some sense, `accidental' as it would correspond to very specific values of some parameters involved.
\end{rem}

 \appendix
\section{The case of general monochromatic perturbations}\label{appendix}
In this section we explain which changes are required along the paper, starting from Section \ref{subsec:Laplace}, to address the case of general $\eta, c\in\R$ in \eqref{eq:monochrom}.

We start observing that Propositions \ref{pro:q_lap} and \ref{pro:q_halfplane} are easily seen to hold in general (see also Remark \ref{rem:alpha}). Then, in \eqref{charge-alpha} the coefficient are given by
\[
\beta_0=\frac{\gamma}{2\pi}-\frac{\log 2}{2\pi}-\frac{\imath}{8}+c\,,\qquad \beta_{\pm 1}=\pm\frac{\imath\alpha_0}{2}e^{\pm\imath\eta}\,.
\]
Accordingly, the pole of the coefficients in \eqref{laplace-charge} is given by 
\[
p_{\mathrm s}=\imath e^{2(\log 2-\gamma)-4\pi c}\,.
\]
The operator defined in the first equation in \eqref{operator} now becomes
\begin{equation}\label{newoperator}
(\mathcal{L}(p)\widehat{q}(p))_{n}:=-\frac{2\pi\imath \alpha_0}{\log (p+\imath \omega n)+4\pi\,\beta_{0}}[e^{\imath\eta}\widehat{q}_{n+1}(p)-e^{-\imath\eta}\widehat{q}_{n-1}(p) ]\,.
\end{equation}
A direct inspection of the proofs of Propositions \ref{profondo rosso}, \ref{inferno}, \ref{tenebre} and \ref{nonhosonno} shows that they are unchanged, except for minor modifications. Namely,
\eqref{eq:series} changes to
\[
\sum_{n\in\Z}[\log(\xi+\omega n)-2\log2+2\gamma+c]\vert \widehat{q}_n(p)\vert^2=4\pi\alpha_0\sum_{n\in\Z}\im(e^{\imath\eta}\widehat{q}_{n+1}(p)\widehat{q}_{n}^*(p)),
\]
and the proof follows along the same lines. Similar modifications apply to \eqref{eq:series} and \eqref{eqsi}.

As a general remark, notice that the no-resonance condition in Section \ref{subsubsec:pzero} (and all along the paper) is formally the same, that is 
\[
N\omega\not=-\lambda_0\,,\qquad \forall N\in\N\,,
\]
with the eigenvalue $\lambda_0$ of the Hamiltonian $\mathcal H_{\alpha(0)}$ given by \eqref{eq:eigenv}.

The results of Section \ref{subsec:lefthalfplane} are also unchanged, up to minor modifications, essentially due to the new definition of the operator $\mathcal L$ in \eqref{newoperator}. This is the case, for instance, for Lemma \ref{lem-cisalva}. Furthermore, \eqref{eq:homogeneous-exp} becomes
\[
\widehat{\xi}(p)=-\alpha_0h(p)[e^{\imath\eta}\widehat{\xi}(p+\imath\omega)-e^{-\imath\eta}\widehat{\xi}(p-\imath\omega)]\,.
\]
The functions $\rho$ and $\Omega$ are defined again by \eqref{eq:rho} and \eqref{eq:Omega}, respectively, but now we have
\begin{equation}\label{eq:Nnew}
\mathcal N(\rho):=\frac{\alpha_0 he^{-\imath\eta}}{1+\alpha_0h e^{\imath\eta}\rho(\cdot+\imath\omega)}
\end{equation}
and
\begin{equation}\label{eq:Mnew}
\mathcal M(\Omega)=- \frac{\alpha_0h(\cdot-\imath\omega)e^{\imath\eta}}{1-\alpha_0 h(\cdot-\imath\omega)e^{-\imath\eta}\Omega(\cdot-\imath\omega)}\,.
\end{equation}
Then the proof of Lemma \ref{lem:contraction} is basically unchanged, as well as that of Corollary \ref{cor:extension_rhoOmega}. Notice that \eqref{eq:L^2} becomes
\begin{multline}
(\cL^2(\alpha_0,p)\widehat{\xi}(p))_n=\alpha^2_0h(p+\imath\omega n)
\left\{h\big(p+\imath\omega(n+1)\big)e^{\imath2\eta}\widehat{\xi}_{n+2}(p)+\right.\\
\left.-\left[h\big(p+\imath\omega(n+1)\big)+h\big(p+\imath\omega(n-1)\big)\right]\widehat{\xi}_n(p))+h\big(p+\imath\omega(n-1)\big)e^{-\imath2\eta}\widehat{\xi}_{n-2}(p)\right\}\,,
\end{multline}
but the proof of Prop. \ref{pro:distance_poles} is the same.

By inspection of the proof one sees that, despite the new definition of the operators in \eqref{eq:Nnew} and \eqref{eq:Mnew}, the poles of the resolvent $(\ID-\mathcal L(\alpha_0,p))^{-1}$ are determined imposing
\[
F(z,\alpha_0)=e^{\imath\eta}(\frac{z}{2p_{\mathrm{s}}\pi\imath}+\frac{2\pi\imath}{\log(p_{\mathrm{s}}-\imath\omega)+4\pi\beta_0+\mathcal{O}(\alpha^2_0)})+\mathcal{O}(\alpha_0^2)=0\,,
\]
so that the exponential factor can be dropped, and we recover \eqref{eq:F-expansion}. Of course, now the exact location of poles also depend on the parameter $c\in\R$, as $\beta_0$ does.

Concerning Section \ref{subsec:poleslargealpha}, it is easy to see that it does not need to be modified as the results presented there rely on abstract arguments. Last, Section \ref{sec:ionization} also follows along the same lines as in the particular case $\eta=0,c=0$.


\end{document}